\documentclass[a4paper,10pt]{article}

\usepackage[english]{babel}
\usepackage[T1]{fontenc}
\usepackage[utf8]{inputenc}
\usepackage{units}
\usepackage[margin=1.25in]{geometry}

\usepackage{amsmath}
\usepackage{afterpage}

\usepackage{graphicx}
\usepackage{subfigure}
\usepackage{wrapfig}
\usepackage{indentfirst}
\usepackage{latexsym}
\usepackage{sidecap}
\usepackage{booktabs}
\usepackage{amsthm}
\usepackage{amssymb}
\usepackage{tabularx}
\usepackage{setspace}
\usepackage{multirow}
\usepackage{mathrsfs}
\usepackage{textcomp}
\usepackage{braket}
\usepackage{bbm}
\usepackage{colortbl}
\usepackage{bbm}
\usepackage{comment}
\usepackage[colorinlistoftodos]{todonotes}
\usepackage{breakcites}

\newcommand{\E}[1]{\mathbb{E}\left  [  #1 \right ] }
\newcommand{\EE}[1]{\mathbb{E}_t\left  [ #1  \right ] }
\newcommand{\R}{\mathbb{R}}

\newcommand{\QQ}{\mathbb{Q}}

\newcommand{\RR}{\mathbb{R}}
\newcommand{\N}{\mathbb{N}}

\newcommand{\sigep}[1]{\sigma^{\epsilon}\left (#1 \right )  }
\newcommand{\Xe}{X_t^{\epsilon}}

\newcommand{\Pow}[2]{\left [#1\right ]_{#2} }
\newcommand{\Ind}[1]{\mathbbm{1}_{\left [#1\right ]}}
\newcommand{\Asi}{0 \leq \epsilon \leq \epsilon_0}
\newcommand{\stlim}{st-lim}
\newcommand{\Indicator}[1]{{\II_{#1}}}
\newcommand{\II}{\ensuremath{\mathbbm{1}}}

\newcommand{\K}[1]{{\mathcal{K}_{#1}}}

\def\bmat{\left[ \begin{array}}
\def\emat{\end{array} \right]}

\newcommand{\EEGr}{\ensuremath{\mathbb{E}}}

\newcommand{\Expectrn}[1]{{\EEGr^{\mathbb{Q}}\left[{#1}\right]}}

\newtheorem{theorem}{theorem}[section]
\newtheorem{remark}[theorem]{remark}
\newtheorem{Lemma}[theorem]{Lemma}
\newtheorem{proposition}[theorem]{proposition}

\newtheorem{Hyp}[theorem]{Hypothesis}
\newtheorem{Ex}{example}[section]

\begin{document}

\date{}
\title{Asymptotic expansion for some local volatility models arising in finance}

\author{Sergio Alberverio$^1$   ,   Francesco Cordoni$^2$ \\
        Luca Di Persio$^3$   ,     Gregorio Pellegrini$^4$\\
        $^1$University of Bonn - HCM , BiBoS, IZKS \\
Bonn, Germany\,\\
email: albeverio@iam.uni-bonn.de\\[2pt]
$^2$University of Verona - Dept. of Computer Science \\
Strada le Grazie, 15, Verona, Italy\\
email: francescogiuseppe.cordoni@univr.it\\[2pt]
$^3$University of Verona - Dept. of Computer Science \\
Strada le Grazie, 15, Verona, Italy\\
email: luca.dipersio@univr.it\\[2pt]
$^4$Intern at Gruppo Generali, Italy \\
Trieste, Italy\\
email: gregorio.pellegrini@gmail.com
}

\maketitle

\begin{abstract}
In this paper we study the  small noise asymptotic expansions for certain classes  of local volatility models arising in finance. We provide explicit expressions for the involved coefficients as well as accurate estimates on the remainders. Moreover, we perform a detailed numerical analysis, with accuracy comparisons,
 of the obtained results by mean of the standard Monte Carlo technique as well as exploiting the polynomial Chaos Expansion approach.
\end{abstract}

\section{Introduction}
In the present paper we shall provide small noise asymptotic expansions for some local volatility models (LVMs) arising in finance. Our approach is based on the rigorous results on asymptotic expansions for solutions of finite dimensional SDE's obtained in \cite{Alb} (following the approach proposed in \cite[Sec.6.2]{Gar}); some extensions to a class of SPDE's and infinite dimensional SDE's have been presented in \cite{Alb3,Alb2a,Alb2b}. In particular we consider underlyings whose behavior is characterized by a stochastic volatility term of $small$ amplitude $\epsilon$ with respect to which  we perform a formal, based on \cite[Sec. 6.2]{Gar}, resp. asymptotic, based on \cite{Alb}, expansion. The latter implies that the  equation characterizing the particular LVM of interest is approximated  by a finite recursive system of a number $N$ of linear equations with random coefficients. We then exploit the solutions of the latter system to provide a formal, resp. an asymptotic, approximation of smooth functions of the original solution for the particular LVM of interest. In a similar way we derive the corresponding approximation for the expected value of the related option price in a {\it risk neutral setting}. Errors estimates and explicit expressions for the involved approximations are also provided for some specific cases, together with a detailed   numerical analysis.

We would like to recall that LVMs are commonly used to analyse options markets where the underlying volatility strongly depends on  the level of the underlying itself. Let us mention that although time-homogeneous local volatilities are supposedly inconsistent with the dynamics of the equity index implied volatility surface, see, e.g., \cite{Mand}, some authors, see, e.g., \cite{Cre}, claim that such models provide the best average hedge for equity index options.

Let us also note that, particularly during recent years, different asymptotic expansions approaches to other particular problems in mathematical finance have been  developed, see, e.g.,  \cite{Lip,Bay,Ben,BGM,Bre,CDP,Fou,Gou,Fuj,Gat,Gu,Kus,Lu,ST,Tak,UchYos,Yos}, see also \cite{AlbH,Imk,Pes} for applications of similar expansion to other areas.

The present paper is organized as follows: in Sect. \ref{SEC:AS} the basic general asymptotic expansions approach, based on \cite{Alb} is presented. Then, in Sec. \ref{optionpriceapproximation} we apply the aforementioned results to important examples in  financial mathematics. In particular in Sec. \ref{optionpriceapproximation} we study a perturbation up to the first order around the Black-Scholes model as well as a correction with jumps for the case of a generic smooth volatility function $f$. We then give more detailed results for the case of an exponential volatility function $f$,in \ref{SEC:ExpCorreNoJ} with Brownian motion driving, in \ref{SEC:JumpExpCorr} with an additional jump term. In \ref{SEC:PolCorr} we shall present detailed corresponding
results for the case of a polynomial volatility function $f$, in \ref{SEC:PolJ}
we treat the case of corrections for $f$ being a polynomial and the
noise containing jumps. 
 To validate our expansions we present their numerical implementations obtained by exploiting the {\it Polynomial Chaos Expansion} approach as well as the standard {\it Monte Carlo} technique, also providing a detailed comparison between the two implementations in terms of accuracy.


\section{The asymptotic expansion}\label{SEC:AS}
\subsection{The general setting}

We shall consider the following stochastic differential equation (SDE), indexed by a parameter $\epsilon \geq 0$
\begin{equation}\label{EQN:SPDE}
\begin{cases}
d \Xe = \mu^\epsilon \left (\Xe\right )dt + \sigep{\Xe}dL_t \;,\\
X^{\epsilon}_0 = x^{\epsilon}_0 \in \RR, \quad t \in [0,\infty)
\end{cases}\; ;
\end{equation}
where $L_t$, $t \in [0,\infty)$, is a real--valued L\'{e}vy process of jump diffusion type, subject to some restrictions which will be specified later on and $\mu^\epsilon:\RR^d \to \RR$, $\sigma^\epsilon: \RR^d \to \RR^{d \times d}$ are Borel measurable functions for any $\epsilon \geq 0$ satisfying some additional technical conditions in order to have existence and uniqueness of strong solutions, e.g., locally Lipschitz and sublinear growth at infinity, see, e.g., \cite{App,Arn,MaRu,Sko,Imk,Shr}. If the L\'{e}vy process $L_t$ has a jump component, then $\Xe$ in eq. \eqref{EQN:SPDE} has to be understood as $X^\epsilon_{t-}:= \lim_{s \uparrow t}X^\epsilon_s$, see, e.g., \cite{MaRu} for details. 

\begin{Hyp}\label{HYP:1}
Let us assume that:
\begin{description}
\item[(i)] $\mu^\epsilon,\sigma^{\epsilon} \in C^{k+1}(\RR)$ in the space variable, for any fixed value $\epsilon \geq 0$ and for all $k \in \N_0:= \N \cup \{0\}$;\\
\item[(ii)] the maps $\epsilon \mapsto \alpha^\epsilon(x)$, where $\alpha=\mu$, $\sigma$, are in $C^{M}(I)$ in $\epsilon$, for some $M \in \N$, for every fixed $x \in \RR$ and where $I := [0,\epsilon_0]$, $\epsilon_0>0$.
\end{description}
\end{Hyp}

Our goal is to show that under Hypothesis \ref{HYP:1} and some further smoothness conditions on $\mu^\epsilon$ and $\sigma^\epsilon$ (needed for the construction of the random coefficients $X^i_t$, $i=0,1,\dots,N$ appearing in \eqref{EQN:XExp} below), a solution $\Xe$ of equation \eqref{EQN:SPDE} can be represented as a power series with respect to  the parameter $\epsilon$, namely 
\begin{equation}\label{EQN:XExp}
\Xe = X_t^0 + \epsilon X_t^1 + \epsilon^2 X_t^2 + \dots+ \epsilon^N X_t^N + R_N(t,\epsilon)\, ,
\end{equation}
where $X^i:[0,\infty) \to \RR$, $i=0,\ldots,N$, are continuous functions, while $|R_N(t,\epsilon)| \leq C_N(t) \epsilon^{N+1}$, $\forall N \in \N$ and $\epsilon \geq 0$, for some $C_N(t)$ independent of $\epsilon$, but in general dependent of randomness, through $X^0_t,X_t^1,\dots,X^N_t$. For $n \in \mathbb{N}$, the functions $X^i_t$ are determined recursively as solutions of random differential equations in terms of the $X^j_t$, $j \leq i-1$, $\forall i\in \left\{1,\ldots,N\right\}$. 

Before giving the proof of the validity of the expression in eq. \eqref{EQN:XExp}, let us recall the following result, see, e.g., \cite{Gia}.
\begin{Lemma}\label{THM:ExpTaylor}
Let $f$ be a real (resp. complex) valued function in $C^{M+1}\left (B(x_0,r)\right )$, $r >0$, $x_0 \in \RR$ for some $M \in \N_0$, where $(B(x_0,r)$ denotes the ball of center $x_0$ and radius $r$.

Then for any $x \in B(x_0,r)$ the following Taylor expansion formula holds
\[
f(x) = \sum_{p=0}^M \frac{D^p f(x_0)}{p!}(x-x_0)^p+ R_M\left (D^{M+1}f(x_0,x)\right )\;,
\]
with $D^p f(x_0) := \left .D^p f(x)\right |_{x=x_0}$ the $p-$th derivative at $x_0$ and
\[
R_M\left (f^{(M+1)}(x_0,x)\right ) := (x-x_0)^{M+1}C_M(x_0,x)\,,
\]
with
\[
C_M(x_0,x) := \frac{M+1}{(M+1)!}\int_0^1 (1-s)^M D^{M+1} f(x_0 + s(x-x_0))ds\; .
\]
We have

{\footnotesize
\begin{equation}\label{EQN:CM}
|C_M(x_0,x)|\leq \frac{M+1}{(M+1)!}\int_0^1 (1-s)^M \sup_{x \in B(x_0,r)}|D^{M+1} f(x_0 + s(x-x_0))|ds =: \tilde{C}_M(x_0)< +\infty
\end{equation}
}

and also

{\footnotesize
\[
|R_M\left (f^{(M+1)}(x_0,x)\right )| \leq |C_M(x,x_0)| |x-x_0|^{M+1} \leq \tilde{C}_M(x_0) |x - x_0|^{M+1} \;,\; M \in \N_0\, .
\]
}

\end{Lemma}

With this lemma in mind, let us then consider a function $f: \RR^+ \times \RR \to \RR$, and $f_\epsilon(x) := f(\epsilon,x)$, $\epsilon\geq 0$, $x \in \RR$. If we then suppose that for any fixed $x \in \RR$, $f$ is of class $C^{K+1}(I)$ in $\epsilon$ for some $K \in \N_0$, $I= [0,\epsilon_0]$, $\epsilon_0 >0$, we can write the Taylor expansion of $f$ around $\epsilon=0$, w.r.t. $\epsilon \in I$ for any fixed $x \in \RR$, as follows
\begin{equation}\label{EQN:Estimatef}
f_\epsilon (x) = \sum_{j=0}^K f_j(x) \epsilon^j + R_K^{f_\epsilon} (\epsilon,x)\; ,
\end{equation}
where $f_j$ is the $j-$th coefficient in the expansion provided by Lemma \ref{THM:ExpTaylor}, while $\sup_x |R_K^{f_\epsilon}(\epsilon,x)|\leq C_{K,f} \epsilon^{K+1}$ for some $C_{K,f}>0$, independent of $\epsilon$. Assume in addition that $ x \mapsto f_j(x)$ are in $C^{M+1}$, $j=0,\dots,K$, for some $M \in \N_0$, then, applying Lemma \ref{THM:ExpTaylor} to the function $f_j$ in $B(x_0,r)$, $r > 0$, we obtain
\begin{equation}\label{EQN:ExpSigFinale}
f_\epsilon(x) = \sum_{j=0}^K \epsilon^j \left [\sum_{\gamma =0}^{M}\frac{D^\gamma f_j(x_0)}{\gamma!}(x-x_0)^\gamma + R_M(f_j^{(M+1)}(x_0,x))\right ] + R^{f_\epsilon}_K(\epsilon,x) \; ,
\end{equation}
with $R_M(f_j^{(M+1)}(x_0,x))$ estimated as in Lemma \ref{THM:ExpTaylor} (with $f_j$ replacing $f$) and $R^{f_\epsilon}_K(\epsilon,x)$ as in \eqref{EQN:Estimatef}. 

Let us now take $x = x(\epsilon)$ assuming $\epsilon \mapsto x(\epsilon)$ in $C^{N+1}$, with $0 \leq \epsilon \leq \epsilon_0$, $0 < \epsilon_0 < 1$ and $x(0) = x_0 \in \RR$. Then by Lemma \ref{THM:ExpTaylor}
\begin{equation}\label{EQN:ExpansionXEps}
x(\epsilon) = \sum_{j=0}^N \epsilon^j x_j + R_N^x(\epsilon), \quad N \in \N_0, \quad x_j \in \R, j=0,1,\dots,N \; ,
\end{equation}
with $f$ replaced by $x$, $M$ replaced by $N$, $x$ by $\epsilon$, $x_0$ by $0$, $D^{M+1} (f(\cdot))$ by $f^{(M+1)}(\cdot)$ and $R_M(f^{(M+1)}(x_0,x))$ by $R_N^x(\epsilon)$. In particular 
\begin{equation}\label{EQN:EstRN}
|R_N^x(\epsilon)| \leq \tilde{C}_N(0)\epsilon^{N+1}\; ,
\end{equation}
with $\tilde{C}_N(0)$ independent of $\epsilon$.

Plugging \eqref{EQN:ExpansionXEps} into \eqref{EQN:ExpSigFinale} we get

{\footnotesize
\begin{equation}\label{EQN:ExpFEpsFin}
\begin{split}
f_\epsilon(x(\epsilon))&= \sum_{j=0}^K \epsilon^j \left [ \sum_{\gamma =0}^{M} \frac{D^\gamma f_j(x_0)}{\gamma!}\left (x(\epsilon) - x_0 \right )^\gamma + R_M\left (f_j^{(M+1)}(x_0,x(\epsilon))\right ) \right ]+ R_K^{f_{\epsilon}}(\epsilon,x(\epsilon)) =\\
&= \sum_{j=0}^K \epsilon^j\left [\sum_{\gamma\leq M}\frac{D^\gamma f_j(x_0)}{\gamma!} \left (\sum_{k=1}^N \epsilon^k x_k + R_N^x(\epsilon)\right )^\gamma + R_M\left (f_j^{(M+1)}(x_0,x(\epsilon))\right ) \right ]\\
&+ R_K^{f_{\epsilon}}(\epsilon,x(\epsilon))\; .
\end{split}
\end{equation}
}

The estimates on $R_M$, $R_K^{f_\epsilon}$ and $R_N^x$ have been given above in Lemma \ref{THM:ExpTaylor}, resp. after \eqref{EQN:Estimatef}, resp. \eqref{EQN:EstRN}.

By Newton's formula we have that, $\forall \, \gamma \in \N_0$, the following holds

{\footnotesize
\begin{equation}\label{EQN:Newon}
\left (\sum_{j=1}^N \epsilon^j x_j + R_N^x(\epsilon)\right )^\gamma = \sum^\gamma_* \frac{\gamma!}{\gamma_1!\dots \gamma_{N+1}!}\epsilon^{\gamma_1+2\gamma_2 +\dots + N\gamma_N}x_1^{\gamma_1}\dots x_N^{\gamma_N}(R_N^x(\epsilon))^{\gamma_{N+1}}\;,
\end{equation}
}

where we have used the notation
\[
\sum_*^\gamma = \sum^{\gamma}_{\substack{
   \gamma_1,\dots,\gamma_{N+1} = 0 \\
   \gamma_1+ 2 \gamma_2 +\dots+N\gamma_{N} + \gamma_{N+1} = \gamma
   }}
   \, ;
\]
hence using \eqref{EQN:Newon} to rewrite \eqref{EQN:ExpFEpsFin} we obtain the following.

\begin{Lemma}\label{LEM:Lemma}
If, for $0 \leq \epsilon < \epsilon_0$, $\epsilon \mapsto x(\epsilon)$ is in $C^{N+1}(I)$, $I= [0,\epsilon_0]$, and $\epsilon \mapsto f_\epsilon(y)$ is $C^{K+1}(\RR)$ in $\epsilon \in I$ and for any $y \in \R$, $y \mapsto f_\epsilon(y)$ is in $C^{M+1}$, the following expansion in powers of $\epsilon$ holds:

{\footnotesize
\begin{equation}\label{EQN:AsympExpa}
\begin{split}
f_\epsilon(x(\epsilon)) =& \sum_{j=0}^K \epsilon^j \left [\sum_{\gamma =0}^{M} \frac{D^{\gamma}f_j(x_0)}{\gamma!} \sum_*^\gamma \frac{\gamma!}{\gamma_1!\dots \gamma_{N+1}!}\epsilon^{\gamma_1+2\gamma_2 +\dots + N\gamma_N}x_1^{\gamma_1}\dots x_N^{\gamma_N}(R_N^x(\epsilon))^{\gamma_{N+1}} \right .\\
&+\left . R_M \left (f_j^{(M+1)}(x_0,x(\epsilon))\right ) \right ] + R_K^{f_\epsilon}(\epsilon,x(\epsilon))\;,
\end{split}
\end{equation}
}

The estimates for the remainders are as follow
\[
\begin{split}
|R_N^x(\epsilon)&|\leq \tilde{C}_N(0)\epsilon^{N+1}\; ,\\
R_M\left (f_j^{(M+1)}(x_0,x(\epsilon))\right ) &\leq  \tilde{C}_M(x_0) |x - x_0|^{M+1}\, ,\\
\sup_{x,\, \epsilon} | R_K^{f_\epsilon}(\epsilon,x)| &\leq C_{K,f}\; , 
\end{split}
\]
with $\tilde{C}_N(0)$, $\tilde{C}_M(x_0)$ and $C_{K,f}$ independent of $\epsilon$.
\end{Lemma}

Taking eq. \eqref{EQN:AsympExpa} into account, we can group all the terms with the same power $k \in \N_0$ of $\epsilon$. Calling $\Pow{f_\epsilon (x(\epsilon))}{k}$ the coefficient of $\epsilon^k$, and using $k=j + \gamma$ with $j=0,\dots,K$, $\gamma_1 + 2 \gamma_2+\dots+N\gamma_N=\gamma$ with $\gamma=0,\dots,M$, we have the following, see, \cite{Alb}.
\begin{proposition}\label{PRO:ExpansionCoef}
Let $x(\epsilon)$ be as in \eqref{EQN:ExpansionXEps} let $f_\epsilon$ as in \eqref{EQN:Estimatef} with $f_j \in C^{M+1}$, $j=0,\dots,K$. Then
\[
f_\epsilon (x(\epsilon)) = \sum_{k=0}^{K+M} \epsilon^k \Pow{f_\epsilon (x(\epsilon))}{k} + R_{K+M} (\epsilon) \; ,
\]
with $|R_{K+M} (\epsilon)| \leq C_{K+M} \epsilon^{K+M+1}$, for some constant $C_{K+M}\geq 0$, independent of $\epsilon$, $\Asi$, and coefficients $\Pow{f_\epsilon (x(\epsilon))}{k}$ defined by

{\footnotesize
\[
\begin{split}
\Pow{f_\epsilon (x(\epsilon))}{0} &= f_0 (x_0);\\
\Pow{f_\epsilon (x(\epsilon))}{1} &= D f_0(x_0)x_1 + f_1(x_0);\\
\Pow{f_\epsilon (x(\epsilon))}{2} &= Df_0(x_0)x_2 + \frac{1}{2}D^2f_0(x_0)x_1^2 +Df_1(x_0)x_1 + f_2(x_0);\\
\Pow{f_\epsilon (x(\epsilon))}{3} &= Df_0(x_0)x_3 + \frac{1}{6}D^3f_0(x_0)x_1^3 +Df_1(x_0)x_2 + Df_2(x_0)x_1 +D^2f_1(x_0)x_1^2+f_3(x_0).\\
\end{split}
\]
}

The general case has the following form

{\footnotesize
\begin{equation}\label{EQN:KthTerm}
\Pow{f_\epsilon (x(\epsilon))}{k} = Df_0(x_0)x_k + \frac{1}{k!}D^kf_0(x_0)x_1^k +f_k(x_0) + B^f_k(x_0,x_1,\dots,x_{k-1})\,,\, k = 1,\dots,K+M
\end{equation}
}

where $B_k^f$ is a real function depending on $(x_0,x_1,\dots,x_{k-1})$ only.
\end{proposition}

\begin{remark}\label{11}
We observe that $\left [f_\epsilon(x(\epsilon))\right ]_k$ depends linearly on $x_k$, non linearly in the inhomogeneity involving the coefficients $x_j$, $0 \leq j \leq k-1$ in \eqref{EQN:ExpansionXEps}. If  $x(\epsilon)$ satisfies \eqref{EQN:ExpansionXEps} and both $\mu^\epsilon$ and $\sigma^\epsilon$ have the properties of the function $f_\epsilon$ in \eqref{EQN:Estimatef}, then the coefficients $\mu^\epsilon(x(\epsilon))$ and $\sigma^\epsilon (x(\epsilon))$ on the right hand side of \eqref{EQN:SPDE} can be rewritten in powers of $\epsilon$, for $\Asi$, as follows
\[
\begin{split}
\mu^\epsilon (x(\epsilon)) = \sum_{k=0}^{K_\mu + M_\mu} \Pow{\mu^\epsilon (x(\epsilon))}{k} \epsilon^k + R_{K_\mu +M_\mu}^\mu (\epsilon);\\
\sigma^\epsilon (x(\epsilon)) = \sum_{k=0}^{K_\sigma + M_\sigma} \Pow{\sigma^\epsilon (x(\epsilon))}{k} \epsilon^k + R_{K_\sigma +M_\sigma}^\sigma (\epsilon) \, ;\\
\end{split}
\]
where the natural numbers $K_\alpha$ and $M_\alpha$, $\alpha=\mu,\sigma$ depend on the functions $\mu^\epsilon$, resp. $\sigma^\epsilon$, and
\[
|R^\alpha_{K_\alpha + M_\alpha}(\epsilon)| \leq C_{K_\alpha +M_\alpha}\epsilon^{K_\alpha+M_\alpha+1}\; ,
\]
for some constants $C_{K_\alpha +M_\alpha}$ depending on $C_j$, $j=0,\dots, K_\alpha+M_\alpha$ but independent of $\epsilon$, and $\Asi$.
\end{remark}

\subsection{The asymptotic character of the expansion of the solution $X^\epsilon_t$ of the SDE in powers of $\epsilon$}\label{SEC:AsChar}

\begin{theorem}\label{THM:AsChar}
Let us assume that the coefficients $\alpha^{\epsilon}$, $\alpha=\mu,\sigma$, of the stochastic differential equation \eqref{EQN:SPDE} are in $C^{K_\alpha}(I)$ as functions of $\epsilon$, $\epsilon \in I= [0,\epsilon_0]$, $\epsilon_0 > 0$, and in $C^{M_\alpha}(\RR)$ as functions of $x$. Let us also assume that $\alpha^\epsilon$ are such that there exists a solution $X^\epsilon_t$ in the probabilistic strong, resp. weak sense of \eqref{EQN:SPDE} and that the recursive system of random differential equations
\[
dX^j_t = \Pow{\mu^\epsilon \left (X^\epsilon_t\right )}{j} dt +\Pow{\sigma^\epsilon \left (X^\epsilon_t\right )}{j} dL_t, \quad j=0,1,\dots,N, \,\,\, t \geq 0\, ,
\]
has a unique solution.

Then there exists a sequence $\epsilon_n \in (0,\epsilon_0]$, $\epsilon_0>0$, $\epsilon_n \downarrow 0$ as $n \to \infty$ such that $X^{\epsilon_n}_t$ has an asymptotic expansion in powers of $\epsilon_n$, up to order $N$, in the following sense:
\[
X^{\epsilon_n}_t = X^0_t + \epsilon_n X^1_t + \dots + \epsilon_n^N X^N_t + R_N(\epsilon_n,t)\;,
\]
with
\[
\stlim_{\epsilon_n \downarrow 0} \frac{\sup_{s \in [0,t]}|R_N (\epsilon_n,s)|}{\epsilon_n^{N+1}} \leq C_{N+1}\;,
\]
for some deterministic $C_{N+1} \geq 0$, independent of $\epsilon \in I$, where $st-lim$ stands for the limit in probability.
\end{theorem}
\begin{proof}
We proceed by slightly modifying the proof in \cite{Alb} since we have to take care of the presence of the explicit dependence on $\epsilon$ of the drift coefficient.

We shall use the fact that
\[
T_N(\epsilon,t) := \frac{\left [X^\epsilon_t - \sum_{j=0}^N \epsilon^j X^j_t\right ]}{\epsilon^{N+1}}\;,\; \epsilon \in (0,\epsilon_0]\;,
\]
satisfies a random differential equation of the form

{\footnotesize
\[
\epsilon^{N+1} dT_N(\epsilon,t) = A^{\mu^\epsilon}_{N+1} \left (X^0_t,\dots,X^N_t,R^N(t,\epsilon)\right )dt+A^{\sigma^\epsilon}_{N+1} \left (X^0_t,\dots,X^N_t,R^N(t,\epsilon)\right )dL_t\;,
\]
}

with coefficients $A^{\alpha^\epsilon}_{N+1}$, $\alpha= \mu,\sigma$ given by

{\footnotesize
\[
A^{\alpha^\epsilon}_{N+1} \left (y_0,y_1,\dots,y_N,y\right ) = \left [\alpha^\epsilon\left (\sum_{j=0}^N \epsilon^j y_j + \epsilon^{N+1}y\right )-\sum_{j=0}^N \epsilon^j \alpha_j(y_0,y_1,\dots,y_N)\right ]\; ,
\]
}

with $\alpha_j$, $j=0,1,\dots,N$ the expansion coefficients of $\alpha^\epsilon$ in powers of $\epsilon \in I$.

By Taylor's theorem one proves
\[
\frac{1}{\epsilon^{N+1}} \sup_{s \in [0,t]} |A^{\alpha^\epsilon}_{N+1} \left (X^0_s,\dots,X^N_s,R^N_s(\epsilon)\right )| \leq C_{N+1}, \quad \epsilon \in (0,\epsilon_0]\; ,
\]
for some $C_{N+1} \geq 0$, independent of $\epsilon$, $\Asi$.

From this one deduces that one can find a sequence $\epsilon_n \to 0$ as $n \to \infty$ s.t.
\[
\stlim_{\substack{
   \epsilon_n \downarrow 0 \\
   n \to \infty
   }}
  \frac{1}{\epsilon_n^{N+1}} \sup_{s \in [0,t]} |A^{\alpha^{\epsilon_n}}_{N+1} \left (X^0_s,\dots,X^N_s,R^N_s(\epsilon_n)\right )| 
\]
exists and it is bounded by $C_{N+1}$.

Under some assumptions on $\mu^\epsilon$, $\sigma^\epsilon$ and $L$ it follows then from a theorem by Skorohod, on the continuous dependence of solutions of SDE's on the coefficients, see, e.g. \cite{Sko}, that
\[
\stlim_{\substack{
   \epsilon_n \downarrow 0 \\
   n \to \infty
   }}
   \sup_{s \in [0,t]} |T_N(\epsilon_n,s)|
\]
exists and it is bounded by $C_{N+1}$, which proves the result.

See \cite{Alb} for more details.
\end{proof}

\begin{remark}
In the context where $L_t$ is a standard Brownian motion, results of the above type have been obtained before in connection with Malliavin calculus in \cite{Wat87}, see, e.g. also \cite{Tak99}. In very recent work \cite{ST} have partially extended the result of \cite{Tak99} to the case of a noise with jumps and with a small coefficient only in the Gaussian noise. Note that in our case the small parameter enters in the full volatility in front of the noise $L_t$. Asymptotic expansions in the case of $L_t$ with jumps have also been discussed using PDE methods in \cite{BGM}, see also \cite{Mat}. Here the coefficients appearing in the expansion for the option price are expressed in terms of the greeks.

Also in the work of \cite{PPR}, PDE and Fourier transformation methods are used to handle an expansion of the solution of the Kolmogorov equation associated with processes with stochastic volatility and general jumps terms. Expansions in terms of nested systems of linearized SDE's also occur in \cite{Fou} and \cite{Tak12}.
\end{remark}

\begin{remark}\label{EX:4}
It can be seen that in general the $k-$th equation for $X^k_t$ in Th. \ref{THM:AsChar} is a nonhomogeneous linear equation in $X_t^k$, but with random coefficients depending on $X_t^0,\dots,X_t^{k-1}$ and with a random inhomogeneity depending on $X^k_t$. Thus it has the general form
\begin{equation}\label{EQN:SPDEKOrder}
\begin{split}
d X^{k}_t =& f_k\left (X^0_t,\dots,X_t^{k-1}\right )X_t^{k} dt + g_k\left (X_t^{0},\dots,X_t^{k-1}\right ) dt \\
&+ \tilde{g}_k\left (X_t^0\right ) dL_t + h_k\left (X_t^{0},\dots,X_t^{k-1}\right )X_t^{k} dL_t\;,
\end{split}
\end{equation}
for some continuous functions $f_k,g_k,\tilde{g}_k$ and $h_k$.
\end{remark}

Let us now look at particular cases.

\begin{Ex}\label{EX:2}
Let $\mu^\epsilon= (a + \epsilon b) x$ and $\sigma^\epsilon = (\sigma_0 + \epsilon \sigma_1) x$ with $a,b,\sigma_0$ and $\sigma_1$ some real constants. Applying Proposition \ref{PRO:ExpansionCoef} we get
\begin{equation}\label{EQN:Sys2}
\begin{split}
X^0_t &= x_0 + \int_0^t aX^0_s ds + \int_0^t \sigma_0 X^0_s dL_s\; ,\\
X^1_t &= \int_0^t a X^1_s ds +\int_0^t b X^0_s ds + \int_0^t \sigma_1 X^0_s dL_t + \int_0^t\sigma_0 X^1_s dL_t \; ,\\
X^k_t &= \int_0^t a X^k_s ds +\int_0^t b X^{k-1}_s ds + \int_0^t \sigma_1 X^{k-1}_s dL_s +\int_0^t \sigma_0 X^k_t dL_t, k \geq 2 \; .
\end{split}
\end{equation}
\end{Ex}


If we consider the special case of remark \ref{11} where $\mu^\epsilon(x) =ax+b $, independent of $\epsilon$, $\sigma^\epsilon(x) = cx+\epsilon \tilde{d} x$, for some real constants $a,b,c$ and $\tilde{d}$, independent of $\epsilon$, and where the L\'{e}vy process is taken to be a standard Brownian motion, $L_t = W_t$, then by eq. \eqref{EQN:KthTerm} we have that $X^k_t$ satisfies a linear equation with constant coefficients for any $k \in \N$, thus applying standard results, see, e.g., \cite{Arn}, an explicit solution for $X^k_t$ can be retrieved.

Let us describe this in the case where we have a set of $K$ coupled linear stochastic equations with random coefficients of the form
\begin{equation}\label{EQN:InHomo}
\begin{cases}
d X_t = \left [ A(t) X_t + f(t) \right ]dt + \sum_{i=1}^m \left [B_i(t)X_t + g_i(t)\right ]dW^i_t,\\
X^k_0 = x^k_0 \in \RR, \quad t \geq 0
\end{cases}
\end{equation}
where, $A$ and $B_i$ are $K\times K$ matrices, $f$ and $g_i$ $\RR^K-$valued deterministic functions. All the coefficients $A,B,f$ and $g$ are assumed to be measurable. The solution of equation \eqref{EQN:InHomo} is then given by
\begin{equation}\label{EQN:SolInHomo}
X_t = \Phi(t) \left [x_0+ \int_0^t \Phi^{-1}(s)\left (f(s) -\sum_{i=1}^m B_i(s)g_i(s)\right ) ds +\sum_{i=1}^m \int_0^t g_i(s) dW^i_s\right ]
\end{equation}
where $\Phi(t)$ is the fundamental $K \times K$ matrix solution of the corresponding homogeneous equation, i.e. it is the solution of the problem
\begin{equation}\label{EQN:FundSol}
\begin{cases}
d \Phi(t) =  A(t) \Phi(t) dt +\sum_{i=1}^m B_i(t) \Phi(t)dW^i_t,\\
\Phi(0) = I, \, t \geq 0\, ,
\end{cases}
\; ,
\end{equation}
being $I$ the unit $K\times K$ matrix.

\begin{remark}
In the case where $K=1$ we have that $\Phi$ reduces to a scalar and is given by
\[
\Phi(t) = \exp\left \{ \int_0^t \left (A(s) - \frac{1}{2} B^2(s)\right )ds  + \int_0^t B(s) dW_s \right \} \; .
\]
Still in the case $K=1$ but with a more general noise, i.e. $W_t$ in eq. \eqref{EQN:InHomo} replaced by a L\'{e}vy process composed by a Brownian motion plus $W_t$ a jump component expressed by $\tilde{N}$, eq. \eqref{EQN:FundSol} is replaced by
\begin{equation}\label{EQN:DDExp}
\begin{cases}
d \Phi(t) =  A(t) \Phi(t) dt +B(t) \Phi(t)dW_t + \int_{\RR_0}\Phi(t_-)C(t,x) \tilde{N}(dt,dx) \,,\\
\Phi(0) = I, \quad t \geq 0
\end{cases}
\; .
\end{equation}
with $A$, $B$ and $C$ Lipschitz and with at most linear growth, and where $\tilde{N}(dt,dx)$ is a Poisson compensated random measure to be understood in the following sense: $\tilde{N}(t,A) := N(t,A) -t\nu(A)$ for all $A \in \mathcal{B}(\R_0)$, $0 \not\in \bar{A}$, with $\bar{A}$ the closure of $A$, $N$ being  a Poisson random measure on $\R_+ \times \R_0$ and $\nu(A) := \mathbb{E}(N(1,A)$, while $\R_0 := \R \setminus \{0\}$ and $\int_{\R_0}(|x|^2\wedge 1)\nu(dx) < \infty$, $\nu$ is the L\'{e}vy measure to $\tilde{N}$, see, e.g. \cite{App,Imk,MaRu}. 

Denoting then eq. \eqref{EQN:DDExp} for short as 
\begin{equation}\label{EQN:dX}
d \Phi(t) =  \Phi(t_-) dX(t)\, ,
\end{equation}
with 
\begin{equation}\label{EQN:dXb}
dX(t) = A(t) dt +B(t) dW_t + \int_{\RR_0}C(t,x) \tilde{N}(dt,dx)\, ,
\end{equation}
we have then that the solution to eq. \eqref{EQN:dX} is explicitly given, in terms of the coefficients and noise, and the solution of eq. \eqref{EQN:dXb}, by
\begin{equation}\label{EQN:19}
\begin{split}
\Phi(t) =& \exp\left \{1+ \int_0^t \left (A(s) - \frac{1}{2} B^2(s)\right )ds  + \int_0^t B(s) dW_s \right. \\
& \left. +\int_{\RR_0}C(s,x) \tilde{N}(ds,dx)\right \}\prod_{0<s\leq t} \left (1 + \Delta X_s\right )e^{-\Delta X_s}\; ,
\end{split}
\end{equation}
where $\Delta X(s) := X_s-X_{s_-}$ is the jump at time $s \in (0,t]$. The stochastic process \eqref{EQN:19} is called {\it Dol\'{e}ans-Dade exponential} (or stochastic exponential) and it is usually denoted by $\Phi(t) = \mathcal{E}(X_t)$. The {\it Dol\'{e}ans-Dade exponential} has a wide use in finance since it is the natural extension to the L\'{e}vy case of the standard geometric Brownian motion, see, e.g., \cite{Arn,Gar} for a more extensive treatment of the fundamental solution of the homogeneous equation for systems of linear SDE's and \cite{App} for more details on the {\it Dol\'{e}ans-Dade exponential}.
\end{remark}


\section{Corrections around the Black-Scholes price (with Brownian, resp. Brownian plus jumps)}\label{optionpriceapproximation}

We shall study an asset $S^\epsilon_t$ evolving according to the particular stochastic differential equation (SDE) governing the Black-Scholes (BS) model, with the possible addition of some driving term determined by a compound Poisson process, see, e.g. \cite{BS,Shr}, resp. \cite{BGM,Mer,Alb5}. Our aim is to apply the theory developed in Sec. \ref{SEC:AS} in order to give corrections around the price given by the BS model for an option with terminal payoff $\Phi$ written on the underlying $S^\epsilon_t$ ($\Phi$ is a given real valued function assumed here to be sufficiently smooth). In particular, if we consider the return process defined as $X^\epsilon_t := log S^\epsilon_t$ ($S^\epsilon_t$ being supposed to be strictly positive, at least almost surely) we have that the price $P(t,T)$ at time $t$ of the option with final payoff $\Phi$ with maturity time $T$,  $0 \leq t \leq T$, is given by
\begin{equation}\label{EQN:OptionPrice}
P(t,T) = \mathbb{E}^{\mathbb{Q}}\left [\left .e^{r(T-t)}\Phi(X_T)\right |\mathcal{F}_t\right ] \, ,
\end{equation}
where $\QQ$ is a relevant equivalent martingale measure, called in financial application {\it risk-neutral measure}, $\mathbb{E}^{\mathbb{Q}}\left [ \cdot | \cdot \right ]$ the corresponding conditional expectation given the $\sigma-$algebra $\mathcal{F}_t$ at time $t$ associated with the underlying Brownian motion, $r>0$ is the constant interest rate. We refer to, e.g., \cite{BS,Bri,CIR,Fil,Kim,Shr} for a general introduction to option pricing.

From Theorem \ref{THM:AsChar} and using Lemma \ref{LEM:Lemma} we deduce that $\Phi(X^\epsilon_t)$ has an asymptotic expansion in powers of $\epsilon \in [0,\epsilon_0)$, $\epsilon_0 >0$, in the sense of Theorem \ref{THM:AsChar}, of the form
\begin{equation}\label{EQN:APPHNM}
\Phi(X_t^\epsilon) = \sum_{k=0}^H \epsilon^k \Pow{\Phi(X^\epsilon_t)}{k} + R_H(\epsilon,t) \, ,
\end{equation}
with 
\[
\sup_{s \in [0,t]} |R_H(\epsilon,s)| \leq C_{H+1}(t) \epsilon^{H+1}\;,
\]
for any $H \in \N$ and the coefficients can be computed from the expansions coefficients of $X^\epsilon_t$, as discussed in section \ref{SEC:AS}.

More concretely we will deal with two particular cases. In the first case we have an asset $S^\epsilon$ evolving according to a geometric Brownian motion with a small perturbation in the diffusion. Namely the asset evolves, in a risk neutral setting, according to
\begin{equation}\label{EQN:21a}
\begin{cases}
dS^\epsilon_t = S^\epsilon_t \left [ (\sigma_0 + \epsilon \sigma_1 \bar{f}(S^\epsilon_t))dW_t\right ]\, ,\\
s_0 = s_0\, , t \geq 0\, ,\\
\end{cases}
\, ,
\end{equation}
where $\sigma_0 \not= 0$ and $\sigma_1$ are real constants, $s_0 > 0$, and $W_t$ is a $\mathbb{Q}$ Brownian motion adapted to the filtration $\left (\mathcal{F}_t\right )_t$, $\bar{f}(S^\epsilon_t) := f(X^\epsilon_t)$ with $f$ a given smooth function on $\RR$. In particular the existence and uniqueness of a strong solution to equation \eqref{EQN:21a} follows under the general assumption of $\bar{f} \in C^1$ from \cite[Problem 3.3.2]{McK}. We have assumed $\sigma_0$ and $\sigma_1$ to be time independent for the sake of simplicity. The generalization to time dependent functions is quite immediate, with no complication in the results developed in what follows.

Suppose, for all $t \geq 0$, $S^\epsilon_t >0$ a.s. (which is the case if $\epsilon$ is sufficiently small). Applying It\^{o}'s lemma to $X^\epsilon_t := \log   S^\epsilon_t$, we end up with the following evolution for $X^\epsilon_t$, the return of the asset price
\begin{equation}\label{EQN:NormRetuGenBrown}
X^{\epsilon}_t = x_0- \int_0^t \left [\frac{\sigma_0^2}{2} + \epsilon \sigma_0 \sigma_1f(X_s^{\epsilon} ) +\epsilon^2 \frac{\sigma_1^2 f(X_s^\epsilon)^2}{2} \right ]  ds + \int_0^t \left [\sigma_0  + \epsilon \sigma_1 f(X_s^{\epsilon} )\right ]dW_s \:,
\end{equation}
where we have set $x_0 := log s_0$. 

Applying the results obtained in Sec. \ref{SEC:AS} and expanding eq. \eqref{EQN:NormRetuGenBrown} to the second order in $\epsilon$ we get
\begin{equation}\label{EQN:SysGen1}
\begin{split}
X_t^0 &= x_0 -\frac{\sigma_0^2}{2}t +\sigma_0  W_t, \quad \mbox{ with law }\quad \mathcal{N}\left (x_0+\mu t,\sigma_0^2 t\right ),\\
X_t^1 &= -\int_0^t   \sigma_0 \sigma_1f(X_s^{0} )ds+\int_0^t \sigma_1 f\left (X_s^0\right )  d W_s ,\\
X_t^2 &= - \int_0^t\left (\frac{\sigma_1^2 f(X^0_s)^2}{2} + 2 \sigma_0 \sigma_1f'\left (X_s^0\right ) X_s^1\right ) ds +\int_0^t \sigma_1 f'\left (X_s^0\right )X_s^1 d W_s \:,
\end{split}
\end{equation}
where $\mathcal{N}\left (-\frac{\sigma_0^2}{2} t,\sigma_0^2 t\right )$ denotes the Gaussian distribution of mean $\mu t$ and variance $\sigma_0^2 t$, $f'$ the derivative of $f$.

The second model we will deal with, following \cite{Mer,BGM}, is the previous one with an addition of a small compound Poisson process
\[
Z_t = \sum_{i=1}^{N_t} J_i\, ,
\]
with $N_t$ a standard Poisson process with intensity $\lambda > 0$ and $(J_i)_{i=1,\dots,N_t}$ being independent normally distributed random variables, namely such that
\[
J_i \, \mbox{ has law } \, \mathcal{N}(\gamma,\delta^2)\, ,
\]
for some $\gamma \in \mathbb{R}$ and $\delta >0$.

We thus have that the L\'{e}vy measure $\nu(dz)$ of $Z$ reads as
\[
\nu(dz) = \frac{\lambda}{\sqrt{2 \pi} \delta } e^{- \frac{(z-\gamma)^2}{2 \delta^2}}\, dz\, , \quad z \in \RR\, ,
\]
and the cumulant function of $Z$ is
\[
\kappa (\zeta) = \lambda \left (e^{\gamma \zeta + \frac{\delta^2 \zeta^2}{2}}-1\right )\, .
\]

In particular we assume the asset $S^\epsilon$ to evolve according to a geometric L\'{e}vy process with a small perturbation in the diffusion. Namely the asset evolves, in a risk neutral setting, according to
\begin{equation}\label{EQN:21ab}
\begin{cases}
dS^\epsilon_t = S^\epsilon_t \left [ (\sigma_0 + \epsilon \sigma_1 \bar{f}(S^\epsilon_t))dW_t + \epsilon \sum_{i=1}^{N_t} J_i\right ] \, ,\\
S^\epsilon_0 = s_0 > 0\, , t \geq 0\, , \\
\end{cases}
\, .
\end{equation}
Again the existence and uniqueness of a strong solution to equation \eqref{EQN:21ab} can be obtained by arguments similar to the ones used in \cite[Problem 3.3.2]{McK} together with the properties of $\sum_{i=1}^{N_t} J_i$.


Proceeding as above, and applying It\^{o}'s lemma to $X^\epsilon_t := \log   S^\epsilon_t$, we have that the log-return process $X^\epsilon_t$ evolves according to
\begin{equation}\label{EQN:NormRetuLevy}
\begin{split}
X^\epsilon_t =& x_0 -\int_0^t \left [\frac{\sigma_0^2}{2} + \epsilon \sigma_0 \sigma_1f(X_s^{\epsilon} ) +\epsilon^2 \frac{\sigma_1^2 f(X_s^\epsilon)^2}{2}\right ] ds+\epsilon \lambda t \left (e^{\gamma + \frac{\delta^2}{2}}-1\right ) \\ 
&+\int_0^t \left (\sigma_0  + \epsilon \sigma_1 f\left (X_s^{\epsilon}\right  )\right )d W_s+ \epsilon  \sum_{i=1}^{N_t} J_i\; ,
\end{split}
\end{equation}
for $\epsilon \in I=[0,\epsilon_0]$, $\epsilon_0>0$.

In the present case it is more tricky to deal with the risk neutral probability measure $\mathbb{Q}$. Under suitable assumptions on the coefficients and noise one can assure the existence (but not necessarily the uniqueness) of an equivalent probability measure $\mathbb{Q}$. We will assume the process \eqref{EQN:NormRetuLevy} to evolve under a risk-neutral measure $\mathbb{Q}$, see, e.g. \cite{App}.

In particular we will use two specific forms for the function $f$, that is an exponential function and a polynomial function. The former is of special interest for its general application to integral transforms, such as Fourier or Laplace transforms, see, e.g. Section \ref{SEC:ExpCorreNoJ}, Remark \ref{REM:Four}. The latter mimics a polynomial volatility process (these type of processes have been widely used in finance since they can be easily implemented, see, e.g. \cite{Car} and reference therein).

\subsection{A correction given by an exponential function}\label{SEC:ExpCorreNoJ}

Let us consider the first model described by equations \eqref{EQN:NormRetuGenBrown} and \eqref{EQN:SysGen1}, i.e. an asset $S^\epsilon$ evolving according to a geometric Brownian motion under the unique risk neutral probability measure $\mathbb{Q}$, recalling that $X^\epsilon_t = log S^\epsilon_t$. Let us first look at the particular case $f(x) = e^{\alpha x}$, for some $\alpha \in \RR$. We take into account the particular case of an exponential function due to the fact that it can be easily extended to the much more general case where the function $f$ can be written as a Fourier transform or a Laplace transform of some bounded measure on the real line, as it will be further discussed in Rem. \ref{REM:Four} below. We then get the following proposition.

\begin{proposition}\label{PRO:ApproxBS3.2}
Let us consider the SDE \eqref{EQN:NormRetuGenBrown} in the particular case where $f(x) = e^{\alpha x}$, for some $\alpha \in \RR_0:= \RR \setminus\{0\}$, $\sigma_0 \in \RR_0$. 

Then the following expansion $X^\epsilon_t = X^0_t + \epsilon X^1_t + \epsilon^2 X^2_t + R_2(\epsilon,t)$ holds, where the coefficients are given by
\begin{equation}\label{EQN:FinalSystem1}
\begin{split}
X_t^0 &= x_0-\frac{\sigma_0^2}{2} t +\sigma_0  W_t, \quad \mbox{ with law }\quad \mathcal{N}\left (x_0\frac{\sigma_0^2}{2} t,\sigma_0^2 t\right );\\
X_t^1 & =\int_0^t   K_\alpha e^{\alpha X_s^{0} }ds+ \frac{\sigma_1}{\alpha\sigma_0}\left (e^{\alpha X^0_t}-1\right ) ; \\
X_t^2 &= C^1_\alpha \int_0^t e^{2 \alpha X^0_s} ds + C^2_\alpha e^{\alpha X_t^0}\int_0^t e^{\alpha X^0_s} ds + C^3_\alpha \int_0^t e^{\alpha X^0_s} ds \\
&+C^4_\alpha \int_0^t e^{\alpha X^0_s}\int_0^s e^{\alpha X^0_r}dr ds + C^5_\alpha e^{2 \alpha X^0_t} +  C^6_\alpha e^{\alpha X^0_t} +C^7_\alpha \, ,
\end{split}
\end{equation}
with 
\[
\begin{split}
K_\alpha:= \sigma_1(\frac{\sigma_0}{2} -\frac{\alpha \sigma_0}{2} -\sigma_0)\,,\, C^1_\alpha:=& -\sigma_1^2 \left ( \frac{5}{2} - \frac{1}{2}+ \alpha +\frac{K_\alpha}{\sigma_0 \sigma_1}\right )\,,\, C^2_\alpha:= K_\alpha  \frac{\sigma_1}{ \sigma_0},\\
C_\alpha^3:=- \sigma_1^2 (\frac{1}{2} + \frac{\alpha}{2} +2)\,,\, C_\alpha^4:=& -K_\alpha \sigma_1 \alpha \left (2 \sigma_0 - \frac{\sigma_0}{2} + \frac{\alpha \sigma_0}{2}\right )\, ,\\
 C^5_\alpha := \frac{\sigma_1^2}{2\alpha \sigma_0^2}\,,\, C^6_\alpha :=& -\frac{\sigma_1^2}{\alpha \sigma_0^2}\,,\, C^7_\alpha := \frac{\sigma_1^2}{2 \alpha \sigma_0^2} \, .
\end{split}
\]
Furthermore $R_2(\epsilon,t)$ satisfies the bound
\[
\stlim_{\epsilon_n \downarrow 0} \frac{\sup_{s \in [0,t]} \left |R_2(\epsilon,s)\right |}{\epsilon^3_n} \leq C_3\, ,
\]
for some subsequence $\epsilon_n \downarrow 0$ and with some constant $C_3 \geq 0$.
\end{proposition}
\begin{proof}
The proof consists in a repeated application of the It\^{o} formula and the stochastic Fubini theorem.

In fact substituting $f(x) = e^{\alpha x}$ into system \eqref{EQN:SysGen1} we immediately obtain
\begin{equation}\label{EQN:SystemExp1}
\begin{split}
X_t^0 &= x_0 \mu t +\sigma_0  W_t, \quad \mbox{ with law }\quad \mathcal{N}\left (x_0+\mu t,\sigma_0^2 t\right );\\
X_t^1 &= -\int_0^t   \sigma_0 \sigma_1e^{\alpha X_s^{0} }ds+\int_0^t  \sigma_1 e^{\alpha X_s^0}  d W_s; \\
X_t^2 &= - \int_0^t \left (\frac{\sigma_1^2}{2}e^{2\alpha X^0_s}+ 2 \sigma_0 \sigma_1 \alpha e^{\alpha X_s^0} X_s^1\right ) ds+\int_0^t \sigma_1 \alpha e^{\alpha X^0_s} X_s^1 d W_s.
\end{split}
\end{equation}

To compute $X^1_t$ we apply It\^{o}'s lemma to the function $g(X_t^0)=e^{\alpha X_t^0}$ to get
\begin{equation}\label{EQN:ItoExp1}
e^{\alpha X^0_t} = 1 + \int_0^t (e^{\alpha X^0_s} \alpha \mu + \frac{\alpha^2}{2}\sigma_0^2 e^{\alpha X^0_s})ds + \int_0^t e^{\alpha X^0_s} \alpha \sigma_0 dW_s\, .
\end{equation}

Expressing the latter integral involving $dW_s$ by the other terms in eq. \eqref{EQN:ItoExp1} and substituting it in the stochastic integral of $X^1_t$ in the system \eqref{EQN:SystemExp1} we get the result for $X^1_t$ in eq. \eqref{EQN:FinalSystem1}.

In order to derive the expression for $X_t^2$ we use again It\^{o}'s lemma, in particular eq. \eqref{EQN:ItoExp1}, getting from \eqref{EQN:SystemExp1}

{\footnotesize
\[
\begin{split}
X_t^2 &= -  \int_0^t \left (\frac{\sigma_1^2}{2}e^{2\alpha X^0_s}+ 2 \sigma_0 \sigma_1 \alpha e^{\alpha X_s^0} X_s^1\right ) ds+\int_0^t \alpha \sigma_1 e^{\alpha X^0_s} X_s^1 d W_s =\\ 
&-\int_0^t \sigma_1^2 (2 \alpha + \frac{1}{2})e^{2\alpha X^0_s} ds +\int_0^t 2 \alpha \sigma_1^2 e^{\alpha X^0_s}ds-\int_0^t \int_0^s 2 K_\alpha \sigma_1\sigma_0  \alpha e^{\alpha X^0_s} e^{\alpha X_r^{0} }dr ds \\
&+\underbrace{\int_0^t \frac{\sigma_1^2\alpha}{ \sigma_0}  e^{2\alpha X^0_s} d W_s}_{(1)} -\underbrace{\int_0^t \frac{\alpha\sigma_1^2}{ \sigma_0} e^{\alpha X^0_s} dW_s}_{(2)} + \underbrace{\int_0^t K_\alpha \alpha \sigma_1 e^{\alpha X^0_s} \int_0^s e^{\alpha X^0_r} dr dW_s}_{(3)} \, .\\
\end{split}
\]
}

For the terms $(1)$ and $(2)$ we use eq. \eqref{EQN:ItoExp1}, resp. It\^{o}'s lemma applied to the function $g(X^0_t) = e^{2 \alpha X^0_t}$, as before to replace the stochastic integral by an integral against Lebesgue measure. In order to treat the term $(3)$ we use the stochastic Fubini theorem, see, e.g. Th. 6.2 in \cite{Fil}, to get
\[
(3)= \frac{K_\alpha \sigma_1}{\sigma_0}\int_0^t \int_0^s \alpha \sigma_0 e^{\alpha X^0_s} e^{\alpha X_r^{0} }dr dW_s = \frac{K_\alpha \sigma_1}{\sigma_0}\int_0^t e^{\alpha X^0_r} \int_r^t \alpha \sigma_0 e^{\alpha X_s^0}  dW_s dr \, .
\]
Using the expression for the integral in $d W_s$ coming from \eqref{EQN:ItoExp1} we then get
\[
\begin{split}
(3) &= \frac{K_\alpha \sigma_1}{\sigma_0}\int_0^t e^{\alpha X^0_r} \int_r^t \alpha \sigma_0 e^{\alpha X_s^0}  dW_s dr  =\\
&= \frac{K_\alpha \sigma_1}{\sigma_0} e^{\alpha X^0_t}\int_0^t e^{\alpha X^0_s} ds -\frac{K_\alpha \sigma_1}{\sigma_0} \int_0^t e^{2\alpha X^0_s} ds  -\frac{K_\alpha \sigma_1}{\sigma_0}(\alpha \mu + \frac{\alpha^2 \sigma_0^2}{2}) \times \\
&\times \int_0^t \int_0^s e^{\alpha X^0_s} e^{\alpha X^0_r} dr ds\, .
\end{split}
\]

Substituting now everything into the original system \eqref{EQN:SystemExp1}, rearranging and grouping the integrals of the same type we get the desired result in \eqref{EQN:FinalSystem1}.

The estimate on the remainder is a consequence of Theorem \ref{THM:AsChar}.
\end{proof}

\begin{remark}
Our aim in Prop. \ref{PRO:ApproxBS3.2} is to discuss in details a particular choice of volatility function around the Black-Scholes one. We obtain explicit formulae for the expansion coefficients, keeping control of the remainder. This expansion can be seen as a particular, but more explicit, case of the one discussed in \cite[Prop. 2.1]{Tak99}. 
\end{remark}

\begin{remark}\label{REM:Four}
The particular choice of $f(x) = e^{\alpha x}$ can easily be extended to any real function which can be written as a Fourier transform, resp. Laplace transform, $f(x) = \int_{\RR_0} e^{i xy}\varrho(d\alpha)$, resp. $f(x) = \int_{\RR_0}e^{\alpha x} \varrho (d\alpha)$, of some positive measure $\varrho$ on $\RR_0$ (e.g. a symmetric probability measure) resp. which has finite Laplace transform. Formula \eqref{EQN:FinalSystem1} holds with $K_\alpha e^{\alpha X^0_\tau}$ replaced by $\int_{\RR_0} K_\alpha e^{i \alpha X^0_\tau} \varrho(d \alpha)$, resp. $\int_{\RR_0} K_\alpha e^{\alpha X^0_\tau} \varrho(d \alpha)$, which are finite if, e.g. $\int_{\RR_0} |K_\alpha| \varrho(d \alpha)<\infty$, resp. $\varrho$ has, e.g., compact support. In fact eq. \eqref{EQN:ItoExp1} gets replaced by 
\begin{equation}
\begin{split}
\int_{\RR_0} e^{\alpha X^0_t} \varrho(d\alpha) =& 1+ \int_{\RR} \left [\int_0^t\left ( e^{\alpha X^0_s} \alpha \mu + \frac{\alpha^2}{2}\sigma_0^2 e^{\alpha X^0_s}\right ) ds \right ] \varrho(d\alpha) \\
&
+ \int_{\RR} \left [\int_0^t e^{\alpha X^0_s} \alpha \sigma_0 dW_s \right ] \varrho(d\alpha)\, .
\end{split}
\end{equation}
By repeating the steps used before and exploiting again the Stochastic Fubini's theorem we get the statements in Prop. \ref{PRO:ApproxBS3.2} extended to these more general cases.
\end{remark}


If we assume the payoff function $x \mapsto \Phi(x)$ to be smooth, $x \in \RR_+$, we can expand $\Phi(X^\epsilon_t)$ in powers of $\epsilon$ using the formulae in Prop. \ref{PRO:ExpansionCoef}. Then, exploiting eq.  \eqref{EQN:APPHNM} with $H=1$, i.e. stopping at the first order, we get
\begin{equation}\label{EQN:23}
\Phi \left (X^{\epsilon}_t\right ) = \Phi(X^0_t) + \epsilon \Phi'(X^0_t) X^1_t +R_1(\epsilon,t)\:,
\end{equation}
with $\sup_{s\in [0,t]} |R_1(\epsilon,s)| \leq \tilde{C}(s)\epsilon^2 $, for some $\tilde{C}$ independent of $\epsilon$ ($\Phi'$ is the derivative of $\Phi$).


Calling $\Phi_1$ the terms on the r.h.s. in eq. \eqref{EQN:23} minus the reminder term $R_1(\epsilon,t)$ we get that the corresponding corrected fair price $Pr^1(0;T)$, up to the first order in $\epsilon$, of an option written on the underlying $S^\epsilon_t:= e^{X^\epsilon_t}$ at time $t=0$ with maturity $T$, reads as follow 
\begin{equation}\label{EQN:ExpOptPrice}
\begin{split}
Pr^1(0;T) &= e^{-rT}\mathbb{E}^{\mathbb{Q}}\left[\Phi_{1} (X^{\epsilon}_T)\right ] =  e^{-rT}\mathbb{E}^{\mathbb{Q}}\left [\Phi(X_T^0) + \epsilon \Phi'(X^0_T) X_T^1   \right ] =\\
&= Pr_{BS}+  \epsilon  e^{-rT}\mathbb{E}^{\mathbb{Q}} \left [\Phi'(X_T^0) X_T^1\right ] \, ,
\end{split}
\end{equation}
where  $Pr_{BS}$ stands for the standard B-S price with underlying $S_t^0 := e^{X_t^0}$, see, e.g \cite{BS}.

This formula yields thus, for a smooth payoff function, the corrected price up to the first order, with an error term related to the "full price" and bounded in modulus by $C_2 \epsilon^2$ for a constant $C_2 \geq 0$ independent of $\epsilon$.

\begin{remark}\label{REM:Smooth}
It is worth to recall that the payoff function usually fails to be smooth such as in the case of European call options where $\Phi(x) = (e^x -K)^+$, $K>0$ being the strike price. The latter payoff function presents namely a point of non differentiability at $e^{X}= K$. Anyhow we can consider a smoothed version of the payoff function, namely $\Phi_h := \Phi * \rho_h$, with $\rho_h$ some smooth kernel s.t. $\Phi_h \to \Phi$ as $h \to \infty$ in distributional sense. With the smoothed payoff function $\Phi_h$, eq. \eqref{EQN:ExpOptPrice} is well defined. In particular the first derivative appearing in eq. \eqref{EQN:ExpOptPrice} is given by a regularized version of $ \Indicator{[x> \ln K]}(x)$. Heuristically, interchanging the limits involved in the expansion with the removing of regularization we can look at $Pr^1(0,T)$ as given by \eqref{EQN:ExpOptPrice} also in the case of the payoff function $\Phi(x) = (e^x -K)^+$, $x \in \RR$, as approximation of the price, with $\Phi' (x )=\Indicator{[x > \ln K]}(x)$ given as above. In the case of smooth coefficients, using methods of \cite{Wat87}, the problem of handling distributional $\Phi$ can be handled efficiently, see, \cite{Tak12}.
\end{remark}

We have the following result.

\begin{proposition}\label{PRO:SecondCOrrectionExpDiff}
Let us consider the particular case of an European call option $\Phi$ with payoff given by $\Phi(X^\epsilon_T)= \max\{e^{X^\epsilon_t} - K,0\}=: \left (e^{X^\epsilon_t}- K\right )_+$, $K$ being the {\it strike price}. Then the approximated price up to the first order, $Pr^1(0;T)$, in the sense of remark \ref{REM:Smooth}, is explicitly given by
\begin{equation}\label{EQN:SecondCOrrectionExpDiff}
Pr^1(0;T) = P_{BS}+ \epsilon \mathcal{K}_1 s_0^{\alpha+1} I_1(s,T,\alpha) - \epsilon \mathcal{K}_2 s_0 N\left (d_1\right )+\epsilon \mathcal{K}_3 s_0^{\alpha+1} N\left (d(2\alpha+1)\right )\,  ,
\end{equation}
with $N(x)$ the cumulative function of the standard Gaussian distribution and

{\footnotesize
\[
\begin{split}
d(\alpha) &= \frac{1}{\sigma_0 \sqrt{T} } \left (log \frac{s_0}{K} +\left ( r- \frac{\sigma_0^2}{2}\alpha\right ) T\right )\, ,\quad d_1 := d(1) \, ,\quad d_2 := (d_1 + \sigma_0\sqrt{T})\, , \\
\mathcal{K}_1 &= K_\alpha e^{ - \frac{\sigma_0^2}{2}T} \,, \quad \mathcal{K}_2 =  \frac{\sigma}{\alpha \sigma_0} , \quad \mathcal{K}_3 = \frac{\sigma_1}{\alpha \sigma_0} e^{ \frac{\sigma_0^2}{2}T \alpha (\alpha+1) + \alpha r T} \, , \\
I_1(s,T,\alpha) & =  \int_0^T e^{\alpha \mu s}\int_{\R\times \R} \Indicator{\left \{x+y>\sqrt{T}d_2
\right\}}e^{\sigma_0 x} e^{(1+\alpha) \sigma_0y} \phi(x,0,T-s)\phi(y,0,s)dxdyds\, ,\\
\end{split}
\]
}

where we have denoted by $\phi(x;\mu,\sigma)$ the density function of the normal distribution with mean $\mu$ and variance $\sigma$, $P_{BS}$ denotes the usual B--S price with underlying $S^0_t = e^{X^0_t}$.
\end{proposition}

\begin{proof}
Given the exponential function $f(x) = e^{\alpha x}$, where $\alpha \in \R$, the approximated price up to the first order, $Pr^1(0; T)$ of an European call option with payoff function $\Phi({X_T^\epsilon}) = \left(e^{X_T^\epsilon}-K\right)_+$ is
\begin{equation}
\begin{split}
Pr^1(0; T) &= P_{BS} + \epsilon e^{-rT}\Expectrn{\Phi'(X_T^0)X_T^1} = \\ &=  P_{BS}+ \epsilon e^{-rT} \Biggl\{
\Expectrn{\Indicator{[X_0^T>\ln(K)]}e^{X_T^0} \int_0^T  K_\alpha e^{\alpha X^0_s}ds}+\\
&-\K{2}\Expectrn{\Indicator{[X_0^T>\ln(K)]}e^{X_T^0}}+\K{2}\Expectrn{\Indicator{[X_0^T>\ln(K)]}e^{X_T^0}e^{\alpha X^0_T}} \Biggl\}\label{sec:eq1} \;,
\end{split}
\end{equation}
where $P_{BS}$ is the standard B-S price with underlying $S^0_t = e^{X_t^0}$. 

Let us first compute the integral
\begin{equation*}
\epsilon e^{-rT} 
\Expectrn{\Indicator{[X_T^0>\ln(K)]}e^{X_T^0} \int_0^T  K_\alpha e^{\alpha X^0_s}ds}
\end{equation*}
By means of Fubini Theorem, we can exchange the expectation with respect to the integration in time so that we obtain
\begin{equation}
\epsilon e^{-rT} K_\alpha \int_0^T
\Expectrn{\Indicator{[X_t^0>\ln(K)]}e^{X_T^0} e^{\alpha X^0_s}}ds  \;.\label{EQN:Ch}
\end{equation}
From the definition of $X_T^0$ and $X_s^0$, for every fixed $0<s<T$, we have
\begin{align*}
X_T^0 &= x_0+\mu T +\sigma_0 W_T \;,\\
X_s^0 &= x_0 +\mu s +\sigma_0 W_s \;,
\end{align*}
are two correlated random variables, by means of the Wiener processes involved. By algebraic manipulation let us define $W_T = W_T-W_s+W_s$, where $X:=W_T-W_s$ is $\mathcal{N}(0,T-s)$ independent with respect to $W_s$. Then 
$X_T^0 =  x_0+\mu T +\sigma_0X+\sigma_0 W_s$ and \eqref{EQN:Ch} becomes

\begin{equation*}
\begin{split}
\epsilon e^{-rT} K_\alpha \int_0^T 
&\Expectrn{\Indicator{\left \{\sigma_0X+\sigma_0 W_s>\ln(\frac{K}{s_0})-\mu T\right\}}e^{(1+\alpha)x_0+\mu T}e^{\alpha \mu s} e^{\sigma_0 X} e^{(1+\alpha) \sigma_0 W_s}}ds  = \\
& = \epsilon e^{-rT} K_\alpha s_0^{(1+\alpha)} e^{rT}e^{-\frac{\sigma_0^2}{2}T}\times \\
&\times \int_0^T e^{\alpha \mu s}\Expectrn{\Indicator{\left \{\sigma_0X+\sigma_0 W_s>\ln(\frac{K}{s_0})-\mu T\right\}}e^{\sigma_0 X} e^{(1+\alpha) \sigma_0 W_s}}ds \;. 
\end{split}
\end{equation*}

The expectation with respect to the risk-neutral measure can be exchanged with the time integration. Moreover by exploiting the independence of $X$ and $W_s$, we get the final result

{\footnotesize
\begin{equation*}
\begin{split}
 &\epsilon K_\alpha s_0^{(1+\alpha)} e^{-\frac{\sigma_0^2}{2}T}\int_0^T e^{\alpha \mu s}\int_{\R\times \R} \Indicator{\left \{x+y>-\sqrt{T}d_2\right\}}e^{\sigma_0 x} e^{(1+\alpha) \sigma_0y} \phi(x,0,T-s)\phi(y,0,s)dxdyds = \\
&= \epsilon s_0^{(1+\alpha)}\K{1}\int_0^T e^{\alpha \mu s}\int_{\R\times \R} \Indicator{\left \{x+y>-\sqrt{T}d_2\right\}}e^{\sigma_0 x} e^{(1+\alpha) \sigma_0y} \phi(x,0,T-s)\phi(y,0,s)dxdyds\;. 
\end{split}
\end{equation*}
}

Then we have from the definition of $X^0_T$
\begin{equation}\label{sec:eq5}
\begin{split}
\mathbb{E}\left [ \Indicator{ [X_0^T>\ln(K)] } e^{ X_T^0 } \right ] &= \int_{x>-d_2} e^{x_0+\mu T+\sigma_0 \sqrt{T}x}\frac{1}{\sqrt{2\pi}}e^{\frac{-x^2}{2}}dx = \\
&= s_0e^{rT}e^{-\frac{\sigma_0^2}{2}T} \int_{x>-d_2} \frac{1}{\sqrt{2\pi}} e^{-\left( \frac{x}{\sqrt{2}}-\frac{\sigma_0 \sqrt{T}}{\sqrt{2}} \right)^2 }e^{\frac{\sigma_0^2T}{2}} dx = \\
&= s_0e^{rT}  \int_{x>-d_2} \frac{1}{\sqrt{2\pi}} e^{-\left(\frac{x}{\sqrt{2}}-\frac{\sigma_0 \sqrt{T}}{\sqrt{2}}\right)^2}dx \;. \\
\end{split}
\end{equation}
By setting $y=x-\sigma_0\sqrt{T}$, the integral in \eqref{sec:eq5} reads as
\begin{equation}
\mathbb{E}\left [ \Indicator{[X_0^T>\ln(K)]}e^{X_T^0} \right ] = s_0e^{rT}  \int_{y>-d_1} \frac{1}{\sqrt{2\pi}} e^{-\frac{y^2}{2}}dx = s_0e^{rT}N(d_1) \;.
\end{equation}
Eventually by multiplying by $-\epsilon e^{-rT}\K{2}$, we obtain
\begin{equation}
-\epsilon e^{-rT}\K{2} \mathbb{E}\left [ \Indicator{[X_0^T>\ln(K)]}e^{X_T^0}\right ] =  -\epsilon \K{2}s_0N(d_1)
\end{equation}

Let us now compute the last term in the bracket $\{ \quad \}$ in \eqref{sec:eq1}. We have

{\footnotesize
\begin{equation}
\begin{split}
\K{2} \mathbb{E}\left [ \Indicator{[X_0^T>\ln(K)]}e^{X_T^0}e^{\alpha X^0_T} \right ] &= \K{2} \int_{x_0+\mu T +\sigma_0\sqrt{T}x>\ln(K)} e^{(1+\alpha)\left(x_0+\mu T +\sigma_0\sqrt{T}x\right)}\frac{1}{\sqrt{2\pi}}e^{\frac{-x^2}{2}} dx  = \\
&= \K{2} \int_{x>-d_2} e^{(1+\alpha)(x_0+\mu T)} e^{(1+\alpha)\sigma_0\sqrt{T}x}\frac{1}{\sqrt{2\pi}}e^{\frac{-x^2}{2}} dx = \\
& = \K{2}s_0^{(1+\alpha)}e^{(1+\alpha)rT}e^{-(1+\alpha)\frac{\sigma_0^2}{2}T}\int_{x>-d_2} e^{(1+\alpha)\sigma_0\sqrt{T}x}\frac{1}{\sqrt{2\pi}}e^{\frac{-x^2}{2}} dx 
\label{sec:eq7b}
\end{split}
\end{equation}
}

The integrand function can be recast as
\begin{align*}
\frac{1}{\sqrt{2\pi}}e^{(1+\alpha)\sigma_0\sqrt{T}x}e^{\frac{-x^2}{2}} =
 \frac{1}{\sqrt{2\pi}}e^{- \left( \frac{x}{\sqrt{2}}-\frac{(1+\alpha)\sigma_0 \sqrt{T}}{\sqrt{2}}\right)^2} e^{\frac{\sigma_0^2}{2}(1+\alpha)^2T} \, .
\end{align*}
By the change of variable $ x \mapsto y = x -(1+\alpha)\sigma_0 \sqrt{T}$, the domain of integration becomes

\begin{equation*}
\begin{split}
y>-d_2 - (1+\alpha)\sigma_0 T & = -\frac{1}{\sigma_0\sqrt{T}}\left ( \ln\left(\frac{K}{s_0}\right) - r T+\sigma_0^2/2 -(1+\alpha)\sigma_0^2 T\right)=\\
& =-\frac{1}{\sigma_0\sqrt{T}}\left ( \ln\left(\frac{K}{s_0}\right) + r T+\frac{\sigma_0^2}{2} (2\alpha+1) T\right) = \\ &= -d(2\alpha+1) \;.
\end{split}
\end{equation*}
Therefore \eqref{sec:eq7b} becomes
\begin{equation*}
\begin{split}
&\K{2} \mathbb{E}\left [ \Indicator{[X_0^T>\ln(K)]}e^{X_T^0}e^{\alpha X^0_T} \right ] =\\
&= \K{2}s_0^{(1+\alpha)}e^{(1+\alpha)rT}e^{-(1+\alpha)\frac{\sigma_0^2}{2}T} e^{\frac{\sigma_0^2}{2}(1+\alpha)^2T}\int_{y>-d(2\alpha+1)} \frac{1}{\sqrt{2\pi}}e^{\frac{-y^2}{2}} dy = \\
& = \K{2}s_0^{(1+\alpha)}e^{(1+\alpha)rT}e^{\alpha(1+\alpha)\frac{\sigma_0^2}{2}T} N(d(2\alpha+1))
\end{split}
\end{equation*}
Eventually by multiplying by $\epsilon e^{-rT}$ we get
\begin{equation*}
\begin{split}
\K{2} \mathbb{E}\left [ \Indicator{[X_0^T>\ln(K)]}e^{X_T^0}e^{\alpha X^0_T} \right ] &= \K{2}s_0^{(1+\alpha)}e^{\alpha rT}e^{\alpha(1+\alpha)\frac{\sigma_0^2}{2}T} N(d(2\alpha+1)) =\\
&= \epsilon \K{3}s_0^{(1+\alpha)}N(d(2\alpha+1))
\end{split}
\end{equation*}

\end{proof}

By Prop. \ref{PRO:SecondCOrrectionExpDiff} we have that the explicit computation of the corrected fair price is reduced to a numerical evaluation of a deterministic integral, which might be more efficient than directly simulating the random variables involved.

\begin{remark}
Our result in Prop. \ref{PRO:SecondCOrrectionExpDiff} covers the case of a perturbation around the classical Black--Scholes model. This is different in this sense from the one discussed in \cite{Tak99}.
\end{remark}

\begin{remark}
We could have also considered the second order perturbation $Pr^2(0;T)$ around the BS price. This is given by
\[
Pr^2(0;T) = Pr^1(0;T) + \epsilon^2  e^{-rT}\mathbb{E}^{\mathbb{Q}} \left [\Phi(X_T^0)' X_T^2\right ]+ e^{-rT}\mathbb{E}^{\mathbb{Q}} \left [\Phi(X_T^0)'' \left (X_T^1\right )^2\right ]\, ,  
\]
with $Pr^1$ the up to first order price in eq. \eqref{EQN:ExpOptPrice}. For the particular case of a European call option we have that $\Phi'' = \delta(X- \log K)e^X + \Ind{X > \log K}e^X$, with $\delta$ the Dirac measure at the origin. Thus the correction up to the second order of the BS price for a European call option reads

\begin{equation}\label{EQN:SecondCOrrectionExpDiff2}
\begin{split}
Pr^2(0;T) &= Pr^1+ \epsilon^2 \mathcal{K}_4s_0^{2\alpha+1} I_1(s,T,2\alpha)+\\
&+ \epsilon^2  \mathcal{K}_5 s_0^{2\alpha+1}I_2(s,T) +\epsilon^2 \mathcal{K}_6 s_0^{\alpha+1} I_1(s,T,\alpha)+\\
&+ \epsilon^2 \mathcal{K}_7 s_0^{2\alpha+1} I_3(r,s,T) + \epsilon^2 \mathcal{K}_8 s_0^{2\alpha+1} N\left (d(-3-4\alpha)\right )+\\
&+\epsilon^2  \mathcal{K}_9 s_0^{\alpha+1} N\left (d(-1-2\alpha)\right )+\epsilon^2 \mathcal{K}_{10} s_0 N(d(1)) \, ,
\end{split}
\end{equation}
with $Pr^1$ as in eq. \eqref{EQN:SecondCOrrectionExpDiff}, the notations as in Prop. \ref{PRO:SecondCOrrectionExpDiff} and

\[
\begin{split}
\mathcal{K}_4 &= \left (C^1_\alpha + 2 K_\alpha \frac{\sigma_1}{\alpha \sigma_0}\right ) e^{- \frac{\sigma_0^2}{2}T}\,,\, \mathcal{K}_5  = C^2_\alpha e^{\alpha r T - \frac{\sigma_0^2}{2}(\alpha+1)T }\,,\, \\
\mathcal{K}_6 &= (C_\alpha^3 + 2 K_\alpha \frac{\sigma_1}{\alpha \sigma_0}) e^{- \frac{\sigma_0^2}{2}T }\,,\, \mathcal{K}_7 = (C_\alpha^4 + 2 K_\alpha^2) e^{ - \frac{\sigma_0^2}{2}T}\,,\, \\
\mathcal{K}_8 &= C^5_\alpha e^{ \frac{\sigma_0^2}{2}T \alpha (2 \alpha+1) + 2 \alpha r T}\,,\, \mathcal{K}_9 = (C^6_\alpha + \frac{\sigma_1}{\alpha \sigma_0}) e^{ \frac{\sigma_0^2}{2}T \alpha (\alpha+1) + \alpha r T}\,,\, \\ 
&\mathcal{K}_{10} = (C^7_\alpha - \frac{\sigma_1}{\alpha \sigma_0})\,,\,\\
I_2(s,T) &= \int_0^T \int_{\RR\times \RR} \Ind{x+y > -\sqrt{T}d(1)} e^{\alpha \mu s + (2\alpha+1) \sigma_0 y+ (\alpha+1)\sigma_0 x}\times \\
&\times \phi(x;y,T-s)\phi(y;0,s) \, dx\,dy\,ds\, ,\\
I_3(r,s,T) =&\int_0^T\int_0^s \int_{\RR\times \RR\times \RR}\Ind{x+y +z> -\sqrt{T}d(1)} e^{\alpha \mu (s+r) + \sigma_0 x+ (1+\alpha) \sigma_0 y + (1+ 2\alpha) \sigma_0 z} \times\\
& \times \phi(x;y,T-s)\phi(y;z,s-r)\phi(z;0,r) \, dx\,dy\, dz\,dr \, ds\, ,\\
\end{split}
\]

\end{remark}

\subsubsection{Numerical results concerning the pricing formula in Prop. \ref{PRO:SecondCOrrectionExpDiff}.}

We will now use the techniques based on the {\it multi-element} {\it Polynomial Chaos Expansion} (PCE) approach, 
 to show the accuracy of the above derived approximated pricing formula in Proposition \ref{PRO:SecondCOrrectionExpDiff}. 

In what follows we will numerically compute the first order correction of the price of an European call option, whose payoff function is $\left(e^{X_T^\epsilon}-K\right)_+$. In particular we focus our attention on the second summand of
\begin{equation}
Pr_1(0;T) = P_{BS}+\epsilon e^{-rT}\Expectrn{\Phi'(X_T^0)X^1_T} \;. \label{sec:eqExpect}
\end{equation}
Also, $X_T^0$ and $X^1_T$ are defined as in Prop. \ref{PRO:ApproxBS3.2}.

The expectation is computed by means of the standard Monte Carlo method, using $10000$ independent realization, and by mean of the multi-element PCE, see, e.g. \cite{BDPG,Crestaux,Ern,Peccati} and references therein, for a detailed introduction to such a method. Indeed, the random variable of interest is 
\begin{equation*}
\Indicator{\{X_0^T(\omega)>\ln(K)\}}\exp(X_T^0)X_T^1 \;.
\end{equation*} 
For both methods we will use the available analytical expression of $X_T^0$ and $X_T^1$, depending on the function $f(x)$. In what follows $D:=\{X_0^T(\omega)>\ln(K)\}$. 

In particular exploiting the linearity of the expectation and the definition of the two random variables involved, \eqref{sec:eqExpect} becomes
\begin{equation}
\Expectrn{\Indicator{D}e^{X_T^0} \int_0^T  K_\alpha e^{\alpha X^0_s}ds}+ K_2\Expectrn{\Indicator{D}e^{X_T^0}e^{\alpha X^0_T}}-K_2\Expectrn{\Indicator{D}e^{X_T^0}}\label{sec:eqExp2}
\end{equation}

Then we perform a {\it multi-element} PCE-approximation of each random variable in \eqref{sec:eqExp2}, setting the degree of the approximation to be $p=15$, 
since the degree of precision reached for such approximation seems to be sufficient. For higher degree the computational costs increase as well as numerical fluctuations, as witnessed exploiting the {\it Non Intrusive Spectral Projection} (NISP) toolbox developed within the {\it Scilab} open source software for mathematics and engineering sciences, becomes relevant for multi-element approximation. It is worth to mention that  {\it multi-element} PCE is nothing else that  a PCE focused on $D$. Moreover the global statistics are given by $D$, scaled by means of the weight $w$.

The numerical values of the parameters are gathered in Table \ref{tab:tab1}.

%

\begin{table}[!h]
\centering
\begin{tabularx}{\textwidth}{c|cccccccc}
Parameters &$\alpha$& r& $\sigma_1$ &K &T \\
\hline
Values & 0.1 & 0.03 & 0.15 & 100 & 0.5\\
\end{tabularx}
\caption{Numerical values of the parameters employed in further computations} \label{tab:tab1}
\end{table}  

The fair price is numerically determined for the set of spot prices $s_0 \in \{90,100,110\}$ and volatility value $\sigma_0 \in \{15\%,25\%,35\%\}$.

The PCE computation will be compared with standard Monte-Carlo simulation for the integrals and expansions in \eqref{sec:eqExp2}. The number of independent realizations is set as $10000$. Moreover as benchmark we use the results presented in Proposition 3.1. These data are collected in Tables \ref{tab:tab2}, \ref{tab:tab3}, \ref{tab:tab4}.

%
%

\begin{table}[!h]
\centering
\begin{tabular}{|c|c|ccc|ccc|}
\hline
 &&\multicolumn{3}{c}{$\epsilon = 0.1$}  & \multicolumn{3}{|c|}{$\epsilon = 0.01$} \\ 
\hline
& & Analytical & PCE & standard MC &Analytical & PCE & standard MC \\[1pt]
\hline 
\multirow{2}{*}{$\sigma_0 = 15 \%$} &
 Results & 12.38180 & 12.37737 &12.36010 &2.22240 & 2.22195 & 2.23204\\
 & Error &&4.4374e-03 & 2.3950e-01 && 4.4374e-04 &2.3995e-02 \\ 
  & Time &&0.0580 & 0.3200 && 0.0530 &0.2890 \\ 
\hline
\multirow{2}{*}{$\sigma_0 = 25 \%$} &
 Results & 14.09613 & 14.08919 &14.14155 &4.31567 & 4.31498 & 4.28696\\ 
 & Error &&6.9451e-03 & 1.7882e-01 && 6.9451e-04 &1.7755e-02 \\ 
  & Time &&0.0530 & 0.3060 && 0.0690 &0.4130 \\ 
\hline 
\multirow{2}{*}{$\sigma_0 = 35 \%$} &
 Results & 15.08779 & 15.07774 &15.30850 &6.58042 & 6.57941 & 6.57030\\ 
 & Error &&1.0044e-02 & 1.4500e-01 && 1.0044e-03 &1.4255e-02 \\ 
  & Time &&0.0690 & 0.3460 && 0.0630 &0.3420 \\

 \hline
\end{tabular}
\caption{Numerical values for PCE and MC estimation of eq. \eqref{sec:eqExp2}, for $s_0 =90$, $\alpha = 0.1$, $\sigma_1 = 0.15$, $r=0.03$ and $T=0.5$.} \label{tab:tab2}
\end{table}

\begin{table}[!h]
\centering
\begin{tabular}{|c|c|ccc|ccc|}
\hline
 &&\multicolumn{3}{c}{$\epsilon = 0.1$}  & \multicolumn{3}{|c|}{$\epsilon = 0.01$} \\ 
\hline
& & Analytical & PCE & standard MC &Analytical & PCE & standard MC \\[1pt]
\hline 
\multirow{2}{*}{$\sigma_0 = 15 \%$} &
 Results & 39.39600 & 39.38877 &38.97870 &8.42541 & 8.42468 & 8.46801\\ 
 & Error &&7.2374e-03 & 3.2398e-01 && 7.2374e-04 &3.2376e-02 \\ 
  & Time &&0.0610 & 0.3180 && 0.3160 &0.3260 \\ 
 \hline
\multirow{2}{*}{$\sigma_0 = 25 \%$} &
 Results & 28.38116 & 28.37206 &28.57793 &9.82235 & 9.82144 & 9.82024\\ 
 & Error &&9.0927e-03 & 2.1197e-01 && 9.0927e-04 &2.1097e-02 \\ 
  & Time &&0.0520 & 0.3000 && 0.0590 &0.2860 \\ 
\hline
\multirow{2}{*}{$\sigma_0 = 35 \%$} &
 Results & 25.56320 & 25.55082 &25.60074 &12.03973 & 12.03850 & 12.03580\\ 
 & Error &&1.2374e-02 & 1.6429e-01 && 1.2374e-03 &1.6466e-02 \\ 
  & Time &&0.0530 & 0.3190 && 0.0550 &0.2940 \\

  \hline
\end{tabular}
\caption{Numerical values for PCE and MC estimation of eq. \eqref{sec:eqExp2}, for $s_0 = 100$, $\sigma_1 = 0.15$, $r=0.03$ and $T=0.5$.} \label{tab:tab3}
\end{table}

\begin{table}[!h]
\centering
\begin{tabular}{|c|c|ccc|ccc|}
\hline
 &&\multicolumn{3}{c}{$\epsilon = 0.1$}  & \multicolumn{3}{|c|}{$\epsilon = 0.01$} \\ 
\hline
& & Analytical & PCE & standard MC &Analytical & PCE & standard MC \\[1pt]
\hline 
\multirow{2}{*}{$\sigma_0 = 15 \%$} &
 Results & 69.70042 & 69.69460 &69.68928 &18.07600 & 18.07542 & 18.08538\\ 
 & Error &&5.8153e-03 & 2.6109e-01 && 5.8153e-04 &2.5932e-02 \\ 
  & Time &&0.0560 & 0.4000 && 0.0700 &0.3330 \\ 
\hline
\multirow{2}{*}{$\sigma_0 = 25 \%$} &
 Results & 45.21665 & 45.20739 &45.04317 &17.54193 & 17.54100 & 17.52920\\ 
 & Error &&9.2595e-03 & 2.0951e-01 && 9.2595e-04 &2.0818e-02 \\ 
  & Time &&0.0690 & 0.3460 && 0.0550 &0.3120 \\ 
\hline
\multirow{2}{*}{$\sigma_0 = 35 \%$} &
 Results & 37.87932 & 37.86590 &37.32253 &19.09570 & 19.09436 & 19.11644\\ 
 & Error &&1.3416e-02 & 1.7161e-01 && 1.3416e-03 &1.7169e-02 \\ 
  & Time &&0.0530 & 0.3510 && 0.0660 &0.3240 \\ 
   \hline
\end{tabular}
\caption{Numerical values for PCE and MC estimation of equation 25, for $s_0 = 110$, $\sigma_1 = 0.15$, $r=0.03$ and $T=0.5$.} \label{tab:tab4}
\end{table}

%

\subsection{A correction given by an exponential function and jumps}\label{SEC:JumpExpCorr}

In what follows we extend the results in Sec. \ref{SEC:ExpCorreNoJ} to the second model in Sec. \ref{optionpriceapproximation}. In particular we will consider a correction up to the first order around the BS price (for a European call option) where both diffusive and jump perturbations are taken into account. We consider an asset whose return evolves according to eq. \eqref{EQN:NormRetuLevy} and consider as before the particular case where $f(x) = e^{\alpha x}$, $\alpha \in \RR_0$. Carrying out the asymptotic expansion in powers of $\epsilon$, $\Asi$, and stopping it at the second order we get the following proposition:

\begin{proposition}\label{PRO:PropJump}
Let us assume $X^\epsilon_t$ evolves according to eq. \eqref{EQN:NormRetuLevy} with $f(x) = e^{\alpha x}$, for some $\alpha \in \RR$, then we have the asymptotic expansion up to the second order in powers of $\epsilon$, $\Asi$, $X^\epsilon_t = X^0_t + \epsilon X^1_t + \epsilon^2 X^2_t + R_2(\epsilon,t)$, where the coefficients are given by

\begin{equation}\label{EQN:SysBSPoiDiff}
\begin{split}
X_t^0 &= x_0+  \mu t +\sigma_0  W_t, \quad \mbox{ with law }\quad \mathcal{N}\left (x_0+\mu t,\sigma_0^2 t\right );\\
X_t^1 &= \int_0^t   K_\alpha e^{\alpha X_s^{0} }ds+ \frac{\sigma_1}{\alpha \sigma_0}\left (e^{\alpha X^0_t}-1\right ) +\lambda t \left (e^{\gamma + \frac{\delta^2}{2}}-1\right )+  \sum_{i=1}^{N_t} J_i \,;\, \\
X_t^2 =& C^1_\alpha \int_0^t e^{2 \alpha X^0_s} ds +  C^2_\alpha e^{\alpha X_t^0}\int_0^t e^{\alpha X^0_s} ds +\\
&+ C^3_\alpha \int_0^t e^{\alpha X^0_s} ds + C^4_\alpha \int_0^t e^{\alpha X^0_s}\int_0^s e^{\alpha X^0_r}dr ds \\
&+ C^5_\alpha e^{2 \alpha X^0_t} +  C^6_\alpha e^{\alpha X^0_t} +C^7_\alpha + C^8_\alpha\lambda \left (e^{\gamma + \frac{\delta^2}{2}}-1\right ) \nu(dx)\int_0^t s e^{\alpha X^0_s} ds \\
&-t e^{\alpha X^0_t} \lambda \left (e^{\gamma + \frac{\delta^2}{2}}-1\right )+ \frac{\sigma_1}{\sigma_0} \lambda \left (e^{\gamma + \frac{\delta^2}{2}}-1\right )\int_0^t e^{\alpha X_s^0}ds \\
&+ C_9^\alpha \int_0^t \sum_{i=1}^{N_s} J_i e^{\alpha X^0_s} ds+ \frac{\sigma_1}{\sigma_0} e^{\alpha X_t^0} \sum_{i=1}^{N_t} J_i - \frac{\sigma_1}{\sigma_0} \sum_{i=1}^{N_t} J_i \int_0^t e^{\alpha X_s^0} ds\, ,
\end{split}
\end{equation}
with the constants as in Prop. \ref{PRO:ApproxBS3.2} and 
\[
C^8_\alpha = \frac{\sigma_1}{\sigma_0} \alpha \mu + \frac{\sigma_0 \sigma_1}{2} \alpha^2 - 2 \sigma_0 \sigma_1 \alpha, \quad C^9_\alpha = 2 \sigma_0 \sigma_1 \alpha -\frac{\sigma_1}{\sigma_0} \alpha \mu - \frac{\sigma_0 \sigma_1}{2} \alpha^2 \, . 
\]

\end{proposition}
\begin{proof}
The proof follows from Prop. \ref{PRO:ApproxBS3.2} just taking into account the presence of the Poisson random measure terms and applying It\^{o}'s lemma, together with the stochastic Fubini theorem.
\end{proof}

\begin{remark}
As mentioned in remark \ref{REM:Four} it is easy to extend Prop. \ref{PRO:PropJump} and formula \eqref{EQN:ExpOptPrice} to the case where $f(x)=e^{\alpha x}$ is replaced by $\int_{\RR_0}e^{i \alpha x} \varrho(d\alpha)$, resp. $\int_{\RR_0} e^{\alpha x} \varrho(d\alpha)$, with assumptions corresponding to those in remark \ref{REM:Four}.
\end{remark}

\begin{proposition}\label{PRO:PriceExpSalti}
Let us consider the model described by \eqref{EQN:NormRetuLevy} in the particular case of an European call option $\Phi$ with payoff given by $\Phi(X^\epsilon_T)= \left (e^{X^\epsilon_t}- K\right )_+$. Then the approximated price up to the first order $Pr^1_\nu(0;T)$, in the sense explained in remark \ref{REM:Smooth}, is explicitly given by
\[
\begin{split}
Pr^1_\nu(0;T) = Pr^1+ \epsilon Ts_0 N\left (d_1\right ) \left (e^{\gamma + \frac{\delta^2}{2}}-1\right ) + \epsilon Ts_0 N\left (d_1\right ) \delta \lambda\, ,
\end{split}
\]
where $Pr^1$ is the corrected fair price up to the first order as given in eq. \eqref{EQN:SecondCOrrectionExpDiff} (the notations are as Prop. \ref{PRO:SecondCOrrectionExpDiff}).
\end{proposition}

\begin{proof}
The proof is analogous of the proof of Prop. \ref{PRO:SecondCOrrectionExpDiff} adding the jump process. The claim follows then from the independence of the jump process and of the Brownian motion together with the fact that $\E{\sum_{i=1}^{N_t} J_i}=\delta T \lambda$ as consequence of the definition of $J_i$ in Section \ref{SEC:ExpCorreNoJ}.
\end{proof}

\subsubsection{Numerical results concerning the pricing formula in Prop. \ref{PRO:PropJump}}

We consider numerically the model discussed in Prop. \ref{PRO:PriceExpSalti}, assuming that the $J_i$ are independent and  normally distributed random variable
\begin{equation*}
J_i \sim \mathcal{N}(\gamma,\delta^2) \qquad \gamma = 0.05, \qquad \delta = 0.02 \;,
\end{equation*}
and $\lambda = 2 $. In particular we are aiming at numerically computing the expectations in the second summand of \eqref{sec:eqExpect}, which in the present case reads
\begin{equation}
\begin{split}
&\Expectrn{\Indicator{D}e^{X_T^0} \int_0^T  K_\alpha e^{\alpha X^0_s}ds}+ K_2\Expectrn{\Indicator{D}e^{X_T^0}e^{\alpha X^0_T}}-K_2\Expectrn{\Indicator{D}e^{X_T^0}}\\ +&K_2\Expectrn{\Indicator{D}e^{X_T^0}\lambda T \left( e^{\gamma+\frac{\delta^2}{2}}-1 \right)}+K_2\Expectrn{\Indicator{D}e^{X_T^0}\sum_{i=1}^{N_T}J_i}\;.\label{sec:eqExp3.1}
\end{split}
\end{equation}
By means of independence of the jumps and $\EE{\sum_{i=1}^{N_T}J_i}= \lambda T \delta$, we get
\begin{equation}
\begin{split}
&\Expectrn{\Indicator{D}e^{X_T^0} \int_0^T  K_\alpha e^{\alpha X^0_s}ds}+ K_2\Expectrn{\Indicator{D}e^{X_T^0}e^{\alpha X^0_T}}-K_2\Expectrn{\Indicator{D}e^{X_T^0}}\\ &K_2\Expectrn{\Indicator{D}e^{X_T^0}\lambda T \left( e^{\gamma+\frac{\delta^2}{2}}-1 \right)}+K_2 \lambda T \delta \Expectrn{\Indicator{D}e^{X_T^0}}\;.\label{sec:eqExp3}
\end{split}
\end{equation}
We are going to compute \eqref{sec:eqExp3} by {\it multi-element} PCE-approximations.

The other parameters entering the model are taken from Table \ref{tab:tab1} and the three spot price considered are $s_0 \in \{90,100,110\}$. The results are presented in Tables \ref{tab:tab5}, \ref{tab:tab6}, \ref{tab:tab7}.

\begin{table}[!h]
\centering
\begin{tabular}{|c|c|ccc|ccc|}
\hline
 &&\multicolumn{3}{c}{$\epsilon = 0.1$}  & \multicolumn{3}{|c|}{$\epsilon = 0.01$} \\ 
\hline
& & Analytical & PCE & standard MC &Analytical & PCE & standard MC \\[1pt]

\hline 
\multirow{2}{*}{$\sigma_0 = 15 \%$} &
 Results & 12.51812 & 12.51387 &12.56922 &2.23603 & 2.23560 & 2.21394\\ 
 & Error &&4.2567e-03 & 2.4285e-01 && 4.2567e-04 &2.4069e-02 \\ 
  & Time &&0.0830 & 0.5250 && 0.0860 &0.5160 \\ 
\hline
\multirow{2}{*}{$\sigma_0 = 25 \%$} &
 Results & 14.31171 & 14.30550 &14.15258 &4.33723 & 4.33661 & 4.34650\\ 
 & Error &&6.2145e-03 & 1.8182e-01 && 6.2145e-04 &1.8295e-02 \\ 
  & Time &&0.0850 & 0.5050 && 0.0880 &0.5190 \\ 
\hline
\multirow{2}{*}{$\sigma_0 = 35 \%$} &
 Results & 15.34622 & 15.33806 &15.39336 &6.60626 & 6.60544 & 6.61219\\ 
 & Error &&8.1646e-03 & 1.4758e-01 && 8.1646e-04 &1.4754e-02 \\ 
  & Time &&0.0880 & 0.5630 && 0.0990 &0.5080 \\

 \hline
\end{tabular}
\caption{Numerical values for PCE and MC estimation of equation 25, for $s_0 = 90$, $\alpha = 0.1$,  $\sigma_1 = 0.15$, $r=0.03$, $\lambda = 2$, $\gamma = 0.05$, $\delta = 0.02$ and $T=0.5$.} \label{tab:tab5}

\end{table}

\begin{table}[!h]
\centering
\begin{tabular}{|c|c|ccc|ccc|}
\hline
 &&\multicolumn{3}{c}{$\epsilon = 0.1$}  & \multicolumn{3}{|c|}{$\epsilon = 0.01$} \\ 
\hline
& & Analytical & PCE & standard MC &Analytical & PCE & standard MC \\[1pt]
\hline 
\multirow{2}{*}{$\sigma_0 = 15 \%$} &
 Results & 39.80797 & 39.80128 &39.84062 &8.46660 & 8.46593 & 8.48244\\ 
 & Error &&6.6879e-03 & 3.2715e-01 && 6.6879e-04 &3.2848e-02 \\ 
  & Time &&0.0870 & 0.5490 && 0.0860 &0.5270 \\ 
\hline
\multirow{2}{*}{$\sigma_0 = 25 \%$} &
 Results & 28.78634 & 28.77863 &28.54522 &9.86287 & 9.86209 & 9.86723\\ 
 & Error &&7.7114e-03 & 2.1566e-01 && 7.7114e-04 &2.1567e-02 \\ 
  & Time &&0.0900 & 0.5370 && 0.0910 &0.6180 \\ 
\hline
\multirow{2}{*}{$\sigma_0 = 35 \%$} &
 Results & 25.96991 & 25.96051 &26.21060 &12.08041 & 12.07947 & 12.05190\\ 
 & Error &&9.3989e-03 & 1.6859e-01 && 9.3989e-04 &1.6726e-02 \\ 
  & Time &&0.1070 & 0.5240 && 0.0920 &0.5180 \\

\hline
\end{tabular}
\caption{Numerical values for PCE and MC estimation of equation 25, for $s_0 = 100$, $\alpha = 0.1$,  $\sigma_1 = 0.15$, $r=0.03$, $\lambda = 2$, $\gamma = 0.05$, $\delta = 0.02$ and $T=0.5$.} \label{tab:tab6}

\end{table}

\begin{table}[!h]
\centering
\begin{tabular}{|c|c|ccc|ccc|}
\hline
 &&\multicolumn{3}{c}{$\epsilon = 0.1$}  & \multicolumn{3}{|c|}{$\epsilon = 0.01$} \\ 
\hline
& & Analytical & PCE & standard MC &Analytical & PCE & standard MC \\[1pt]
\hline 
\multirow{2}{*}{$\sigma_0 = 15 \%$} &
 Results & 70.37793 & 70.37303 &70.57558 &18.14375 & 18.14326 & 18.20012\\ 
 & Error &&4.9058e-03 & 2.6286e-01 && 4.9058e-04 &2.6111e-02 \\ 
  & Time &&0.0850 & 0.4990 && 0.0850 &0.5210 \\ 
\hline
\multirow{2}{*}{$\sigma_0 = 25 \%$} &
 Results & 45.81367 & 45.80646 &45.98197 &17.60163 & 17.60091 & 17.59774\\ 
 & Error &&7.2116e-03 & 2.1243e-01 && 7.2116e-04 &2.1281e-02 \\ 
  & Time &&0.0950 & 0.5270 && 0.0910 &0.5180 \\ 
\hline
\multirow{2}{*}{$\sigma_0 = 35 \%$} &
 Results & 38.43779 & 38.42848 &38.04608 &19.15155 & 19.15062 & 19.13383\\ 
 & Error &&9.3058e-03 & 1.7688e-01 && 9.3058e-04 &1.7692e-02 \\ 
  & Time &&0.0910 & 0.5370 && 0.0990 &0.5430 \\

 \hline
\end{tabular}
\caption{Numerical values for PCE and MC estimation of equation 25, for $s_0 = 110$, $\alpha = 0.1$,  $\sigma_0 = 0.15$, $r=0.03$, $\lambda = 2$, $\gamma = 0.05$, $\delta = 0.02$ and $T=0.5$.} \label{tab:tab7}

\end{table}

\subsection{A correction given by a polynomial function}\label{SEC:PolCorr}

Let us consider eq. \eqref{EQN:NormRetuGenBrown} with $f$ a polynomial correction, namely $f(x) = \sum_{i=0}^N \alpha_i x^i$, with $\alpha_i \in \RR$ and $N \in \N_0$. We then get the following proposition.

\begin{proposition}\label{PRO:ApproxBS}
Let us consider the case of the B-S model corrected by a non-linear term given by \eqref{EQN:NormRetuGenBrown} with $f(x) = \sum_{i=0}^N \alpha_i x^i$, for some $\alpha_i \in \RR$, then the expansion coefficients for the solution $X^\epsilon_t$ of \eqref{EQN:NormRetuGenBrown} up to the second order are given by the system
\begin{equation}\label{EQN:32A}
\begin{split}
X_t^0 &= x_0+ \mu t +\sigma_0  W_t, \quad \mbox{ with law }\quad \mathcal{N}\left (x_0+\mu t,\sigma_0^2 t\right );\\
X_t^1 & = \sum_{i=1}^N \tilde{K}_i (X^0_t)^{i+1}- \sum_{i=0}^N \int_0^t   K_i (X^0_s)^i ds + \sigma_1 \alpha_0 W_t; \\
X_t^2 =& \sum_{k=1}^{2N+1}C^1_k (X^0_t)^k - \sum_{k=1}^{2N+1}\int_0^t C^2_k (X^0_s)^k ds+\\
&+ \sum_{i=1}^N \sum_{j=0}^N \int_0^t\int_0^s C_{i,j}^3 (X^0_s)^{i-1}(X_r^0)^j dr ds\\
&+ \sum_{i=1}^N \sum_{j=0}^N (X^0_t)^i\int_0^s C_{i,j}^4 (X_r^0)^j dr \, .
\end{split}
\end{equation}
where the constants are given by
\[
K_i =
\begin{cases}
\sigma_0 \sigma_1 \alpha_i + \frac{\sigma_1}{\sigma_0}\mu \alpha_i + \frac{\sigma_0 \sigma_1}{2} \alpha_{i+1}(i+1) , \quad i \not= 0, i \not= N\, ,\\
\sigma_0 \sigma_1 \alpha_0 + \frac{\sigma_0 \sigma_1 \alpha_1}{2}, \quad i = 0\, ,\\
\sigma_0 \sigma_1 \alpha_N + \frac{\sigma_1}{\sigma_0}\mu \alpha_N, \quad i = N\, ,\\
\end{cases}
\tilde{K}_i= \frac{\sigma_1}{\sigma_0}\frac{\alpha_i}{(i+1)}\, ,
\]
\[
C^1_k = \gamma_k^1+ \gamma_k^2 + \gamma_k^3\, ,
\]
where
\[
\begin{split}
\gamma_k^1 &=
\begin{cases}
\sum_{k=i+j+1}\mu i \alpha_i + \frac{\sigma_1}{\sigma_0}- \frac{\sigma_0 }{2} (i+j+1) , \quad k \not= 1, k \not= 2N\, ,\\
\frac{\sigma_0}{2}, \qquad \qquad \qquad \qquad \qquad \qquad \qquad  \qquad k = 0\, ,\\
\mu \frac{\sigma_1}{\sigma_0}N \alpha_N(\sigma_0 \sigma_1 \alpha_N + \frac{\sigma_1}{\sigma_0}\mu \alpha_N), \qquad \qquad \quad k = 2N\, ,\\
\end{cases}
\\
\gamma_k^2&= 
\begin{cases}
\left (\frac{(-1)^k +1}{2} \frac{\sigma_1^2}{2}\right )\alpha_k^2, \qquad \qquad \qquad \qquad \quad \mbox{ if } 1 \leq k \leq N,\\
0 , \qquad \qquad \qquad \qquad \qquad  \qquad \qquad \qquad \mbox{ otherwise}\, 
\end{cases}
\\
\gamma_k^3&=  \sum_{i+j= k-1} 2 \sigma_0 \sigma_1 \alpha_i i \tilde{K}_j\, ,\\
C_{i,j}^3 &= -
\begin{cases}
\frac{\sigma_1}{\sigma_0} \alpha_1 K_0, \qquad \qquad \qquad \qquad \qquad \qquad \mbox{ if } i=1,\, \, j=0\, ,\\
\frac{\sigma_1}{\sigma_0} i \alpha_i K_j + \frac{\sigma_0 \sigma_1}{2} i \alpha_i K_j (i-1), \qquad \quad \mbox{ otherwise} \,.
\end{cases}
\\
C_{i,j}^4 &= \frac{\sigma_1}{\sigma_0} \alpha_i K_j\, ,\\
\end{split}
\, .
\]
\end{proposition}
\begin{proof}

The proof consists in a series of applications of It\^{o}'s formula and stochastic Fubini theorem, see, e.g. \cite{Fil} Th. 6.2. In fact, substituting $f(x) = \sum_{i=0}^N \alpha_i x^i$ into system \eqref{EQN:SysGen1} we  obtain
\begin{equation}\label{EQN:SystemExp2}
\begin{split}
X_t^0 &= x_0+ \mu t +\sigma_0  W_t, \quad \mbox{ with law }\quad \mathcal{N}\left (x_0+\mu t,\sigma_0^2 t\right );\\
X_t^1 &= -\int_0^t   \sigma_0 \sigma_1 \left (\sum_{i=0}^N \alpha_i (X^0_s)^i \right )ds+\int_0^t \sigma_1 \left (\sum_{i=0}^N \alpha_i (X^0_s)^i \right )  d W_s; \\
X_t^2 =& - \int_0^t \frac{\sigma_1^2}{2}\left (\sum_{i=0}^N \alpha_i (X^0_s)^i \right )^{2} + 2 \sigma_1 \left (\sum_{i=0}^N \alpha_i (X^0_s) \right )' X^1_sds\\
&+\int_0^t  \sigma_1 \left (\sum_{i=0}^N \alpha_i (X^0_s) \right )' X_s^1 d W_s\, .
\end{split}
\end{equation}

To compute $X^1_t$ obtaining eq. \eqref{EQN:32A} we apply It\^{o}'s lemma to the function $g(X_t^0)=\alpha_{i+1}(X_t^0)^{i+1}$ to get
\begin{equation}\label{EQN:ItoExp}
\begin{split}
(X_t^0)^{i+1} &=  \int_0^t \left (\mu (i+1)(X_s^0)^{i}\frac{1}{2}\sigma_0^2 i(i+1) (X_s^0)^{i-1}\right )ds +\\
&+ \int_0^t (X_s^0)^{i} (i+1) \sigma_0 dW_s\, .
\end{split}
\end{equation}

Then, summing up we obtain
\begin{equation}\label{EQN:SumITerms}
\begin{split}
\sum_{i=1}^N \int_0^t (X_s^0)^{i} (i+1) \sigma_0 dW_s &= \sum_{i=1}^N(X_s^0)^{i+1} +\\
&- \sum_{i=1}^N  \int_0^t \left (\mu (i+1)(X_s^0)^{i}+ \frac{1}{2}\sigma_0^2 i(i+1)(X_s^0)^{i-1}\right )ds\,.
\end{split}
\end{equation}

Substituting now eq. \eqref{EQN:SumITerms} into $X^1$ in eq. \eqref{EQN:SystemExp2} we obtain the following
\[
\begin{split}
X^1_t &= \sum_{i=1}^N \frac{\sigma_1}{\sigma_0}\frac{\alpha_i}{(i+1)}(X^0_t)^{i+1} +\\
&-\sum_{i=1}^N \int_0^t \sigma_0 \sigma_1 \alpha_i (X_s^0)^i-\sum_{i=1}^N \int_0^t \mu (i+1) \frac{\sigma_1 \alpha_{i}}{\sigma_0 (i+1)}(X_s^0)^i +\\
&-\sum_{i=1}^N \int_0^t\frac{1}{2}\sigma_0^2 i(i+1)\frac{\sigma_1 \alpha_{i}}{\sigma_0 (i+1)}(X_s^0)^{i-1}ds\, ,
\end{split}
\]
and rearranging the terms we then get the desired result in \eqref{EQN:32A} for $X^1_t$.

Substituting the expression of $X^1_t$ into $X^2_t$ we obtain
\[
\begin{split}
X^2_t &= - \sum_{i=1}^N \int_0^t \frac{\sigma_1^2}{2} \alpha_i^2 (X^0_s)^{2i}ds -\sum_{i,j=1}^N \int_0^t 2 \sigma_0 \sigma_1 \alpha_i i \tilde{K}_j (X^0_s)^{i-1} (X^0_s)^{j+1}ds=\\
&=\sum_{j=0}^N \sum_{i=1}^N \int_0^t \int_0^s 2 \sigma_0 \sigma_1 \alpha_i i K_j (X^0_s)^{i-1} (X^0_r)^{j}dr ds+\\
&+\sum_{i,j=1}^N \int_0^t   \sigma_1 \alpha_i i K_j (X^0_s)^{i-1} (X^0_s)^{j+1} dWs\\
&- \sum_{j=0}^N \sum_{i=1}^N \int_0^t \int_0^s  \sigma_1 \alpha_i i K_j (X^0_s)^{i-1} (X^0_r)^{j}dr dWs\, .
\end{split}
\]
Exploiting again the stochastic Fubini theorem, from eq. \eqref{EQN:SumITerms} and grouping the terms with the same powers we obtain \eqref{EQN:32A}.
\end{proof}

\begin{proposition}\label{PRO:LinCorr}
Let us consider the particular case of $N=1$, i.e. a linear perturbation, namely $f(x) = \alpha_0 + \alpha_1 x$, $\alpha_i \in \RR$, $i=0,1$. Then the terms up to the first order in equation \eqref{EQN:32A} read
\begin{equation}\label{EQN:LinePrice}
\begin{split}
X^0_t &= x_0 + \mu t + \sigma_0 W_t\, ,\\
X^1_t &= \beta_1 t + \beta_2 t^2 + \beta_3 W_t + \beta_4 W_t^2 + \beta_5 t W_t - \int_0^t \beta_6 W_s ds\;,
\end{split}
\end{equation}

with
\[
\begin{split}
\beta_1 &= - \sigma_0 \sigma_1 \alpha_0 - \sigma_0 \sigma_1 \alpha_1 x_0- \frac{\sigma_0 \sigma_1 \alpha_1}{2}\,,\\
\beta_2 &= -\frac{\sigma_0 \sigma_1 \alpha_1 \mu}{2}\,, \quad \beta_3 = \alpha_1 \sigma_0 + x_0 \sigma_1 \alpha_1\, ,\\
\beta_4 &= \frac{\sigma_0 \sigma_1 \alpha_1}{2}, \quad \beta_5 = \sigma_1 \alpha_1 \mu , \quad \beta_6 =  \sigma_1 \alpha_1 \mu  + \sigma_0^2 \sigma_1 \alpha_1  \, .
\end{split}
\]

The first order correction (in the sense discussed in remark \ref{REM:Smooth}), of the price of an European call option $\Phi$ with payoff given by $\Phi(X^\epsilon_T)= \left (e^{X^\epsilon_T}- K\right )_+$ is explicitly given by
\begin{equation}\label{EQN:PrezzoPoli}
\begin{split}
Pr^1(0;T) =& P_{BS} + \epsilon s_0 (\beta_1 + \sigma_0\beta_3 + \beta_4) T N\left (d_1\right )+ \epsilon s_0 (\beta_2 +\sigma_0^2 \beta_4) T^2 N\left (d_1\right )  \\
&+\epsilon s_0(\beta_3 +2 \sigma_0 \beta_4 T+ T \beta_5)  \sqrt{T} \phi\left (-d_1\right ) - \epsilon s_0 \beta_4 T d_1 \phi(d_1)+ \\
&+ \epsilon s_0 T^2 \beta_5 \sigma_0 T^2 N \left (d_1\right )-  \epsilon s_0e^{
 + \frac{\sigma^2_0}{2}T}\beta_6 I(s,T) \, ,
\end{split}
\end{equation}
where the notation is as in Prop. \ref{PRO:SecondCOrrectionExpDiff} and we have denoted for short by $\phi(x)$ the density function of the standard Gaussian law and we have set
\[
I(s,T) = \int_0^T \int_{\RR \times \RR} \Ind{x+y> -\sqrt{T}d_1} e^{\sigma_0(x+y)}y \phi(x;0,T-s) \phi(y;0,s) dx \, dy \, ds\, .
\]
\end{proposition}

\begin{proof}
Let us consider the linear function $f(x) = \alpha_0+\alpha_1x$, where $\alpha_0,\alpha_1 \in \R$.
The approximated price up to the first order, $Pr^1(0; T)$ of an European call option with payoff function $\Phi({X_T^\epsilon}) = \left(e^{X_T^\epsilon}-K\right)_+$ is
\begin{equation}
\begin{split}
Pr^1(0; T) &= P_{BS} + \epsilon e^{-rT}\Expectrn{\Phi(X_T^0)X_T^1} \label{sec:eq9b}
\end{split}
\end{equation}
where $P_{BS}$ is the standard B-S price with underlying $s_0(t) = e^{X_t^0}$. 

In particular we have that $X_T^0$ and $X_1^T$ are defined as
\begin{align}
X_T^0 & = x_0+\mu T+\sigma_0 W_T \\
X_T^1 & = \beta_1 T+\beta_2T^2+ \beta_3 W_T+\beta_4 W_T^2+\beta_5TW_T -\beta_6 \int_0^T W_s ds \;.
\end{align}
By linearity of the expectation, \eqref{sec:eq9b} becomes, collecting the terms with coefficients $\beta_3$ and $\beta_5$,
\begin{equation}\label{EQN:Summand}
\begin{split}
Pr^1(0; T) &= P_{BS} + \epsilon e^{-rT} \Biggl \{\Expectrn{\beta_1 T\Indicator{\left\{X_0^T>\ln(K)\right\}}e^{X_T^0}}+\\
&+\Expectrn{\beta_2T^2\Indicator{\left\{X_0^T>\ln(K)\right\}}e^{X_T^0}} +\Expectrn{\beta_{3,5}^T W_T\Indicator{\left\{X_0^T>\ln(K)\right\}}e^{X_T^0}}+\\
&+\Expectrn{\beta_4 W_T^2\Indicator{\left\{X_0^T>\ln(K)\right\}}e^{X_T^0}}+\Expectrn{\beta_6 \Indicator{\left\{X_0^T>\ln(K)\right\}}e^{X_T^0} \int_0^T W_s ds} \Biggl\} \, ,
\end{split}
\end{equation}
with $\beta_{3,5}^T := \beta_3 + T \beta_5$.

From the definition of $X_T^0$ we have that
\begin{equation*}
\begin{split}
\epsilon e^{-rT}\Expectrn{\beta_1 T\Indicator{\left\{X_0^T>\ln(K)\right\}}e^{X_T^0}} &= \epsilon T \beta_1 s_0 N(d_1)\, , 
\end{split} 
\end{equation*}
and as above we have
\begin{equation*}
\begin{split}
\epsilon e^{-rT}\Expectrn{\beta_2 T^2\Indicator{\left\{X_0^T>\ln(K)\right\}}e^{X_T^0}} &= \epsilon T^2 \beta_2 s_0 N(d_1) 
\end{split} 
\end{equation*}

Concerning the third term in \eqref{EQN:Summand}, we have that,
\begin{equation*}
\begin{split}
&\beta_{3,5}^T \Expectrn{ W_T\Indicator{\left\{X_0^T>\ln(K)\right\}}e^{X_T^0}} =\\
&= \beta_{3,5}^T s_0 e^{rT}e^{-\frac{\sigma_0^2}{2}T}\sqrt{T}\int_{x>-d_2}e^{\sigma_0 \sqrt{T} x}x \frac{1}{\sqrt{2\pi}}e^{-\frac{x^2}{2}}dx = \\
& = \beta_{3,5}^T s_0e^{rT}e^{-\frac{\sigma_0^2}{2}T} \sqrt{T}\int_{x>-d_2}x \frac{1}{\sqrt{2\pi}}e^{-\left(\frac{x}{\sqrt{2}}-\frac{\sigma_0 \sqrt{T}}{\sqrt{2}}\right)^2}e^{\frac{\sigma_0^2}{2}T}dx \, ,
\end{split} 
\end{equation*}
and by setting $y=x-\sigma_0 \sqrt{T}$, we get that the r.h.s. is given by
\begin{equation*}
\begin{split}
&\beta_{3,5}^T  s_0e^{rT}\sqrt{T}\int_{y>-d_1}\left(\sigma_0 \sqrt{T}+y\right) \frac{1}{\sqrt{2\pi}}e^{-\frac{y^2}{2}}dy =\\
&=\beta_{3,5}^T T  s_0e^{rT} \sigma_0 N(d_1)- \beta_{3,5}^T \sqrt{T} s_0e^{rT} \left[\frac{1}{\sqrt{2\pi}}e^{-\frac{y^2}{2}}\right]_{-d_1}^{+\infty} = \\
& =\beta_{3,5}^T T s_0e^{rT}\sigma_0  N(d_1)+\beta_{3,5}^T \sqrt{T}  s_0e^{rT}\phi(-d_1,0,1)\;.
\end{split} 
\end{equation*}
Hence the third term in \eqref{EQN:Summand} reads
\begin{equation*}
\epsilon e^{-rT}\Expectrn{\beta_3 W_T\Indicator{\left\{X_0^T>\ln(K)\right\}}e^{X_T^0}}  = \epsilon \beta_3 T \sigma_0 s_0N(d_1)+\epsilon \beta_3  s_0\sqrt{T}\phi(-d_1,0,1)\;.
\end{equation*}

Exploiting the definition of $X_T^0$ occurring in the fourth term in \eqref{EQN:Summand}, as well as similar algebraic computation as in the previous previous section, we get
\begin{equation*}
\begin{split}
\Expectrn{\beta_4 W_T^2\Indicator{\left\{X_0^T>\ln(K)\right\}}e^{X_T^0}} &= \beta_4 s_0 e^{rT}e^{-\frac{\sigma_0^2 }{2}T}\int_ {x>-d_2} T x^2 e^{\sigma_0 \sqrt{T}x}\frac{1}{\sqrt{2\pi}}e^{-\frac{x^2}{2}}dx = \\
& =  \beta_4 s_0 e^{rT}T \int_ {y>-d_1}  (y+\sigma_0\sqrt{T})^2 \frac{1}{\sqrt{2\pi}}e^{-\frac{y^2}{2}}dy \, .
\end{split} 
\end{equation*}

Developing the square and using the linearity property of the integral we get that the r.h.s. is equal to

\begin{equation*}
\begin{split} 
&\int_ {y>-d_1}  y^2\frac{1}{\sqrt{2\pi}}e^{-\frac{y^2}{2}}dy+ \int_ {y>-d_1} 2\sigma_0\sqrt{T}y\frac{1}{\sqrt{2\pi}}e^{-\frac{y^2}{2}}dy+ \int_ {y>-d_1} \sigma_0^2T \frac{1}{\sqrt{2\pi}}e^{-\frac{y^2}{2}}dy = \\
& = \int_ {y>-d_1}  y^2\frac{1}{\sqrt{2\pi}}e^{-\frac{y^2}{2}}dy+ 2 \sigma_0 \sqrt{T}\phi(-d_1,0,1)+ \sigma_0^2 T N(d_1) \, .
\end{split} 
\end{equation*}

The first term is computed using integration by parts,
\begin{equation*}
\begin{split} 
&\int_ {y>-d_1}  y^2\frac{1}{\sqrt{2\pi}}e^{-\frac{y^2}{2}}dy= -d_1 \phi(d_1)+N(d_1)
\end{split} 
\end{equation*}
therefore
\begin{equation*}
\begin{split}
&\epsilon e^{-rT}\Expectrn{\beta_4 W_T^2\Indicator{\left\{X_0^T>\ln(K)\right\}}e^{X_T^0}} =\\
&= \epsilon \beta_4 s_0 T \left( -d_1 \phi(d_1)+N(d_1)+ 2 \sigma_0 \sqrt{T}\phi(-d_1,0,1)+ \sigma_0^2 T N(d_1)\right) \, .
\end{split} 
\end{equation*}

To compute the fifth term in \eqref{EQN:Summand} we use Fubini theorem to exchange the expectation with the integral with respect to time, getting
\begin{equation}
\beta_6  \int_0^T  \Expectrn{ \Indicator{\left\{X_0^T>\ln(K)\right\}}e^{X_T^0}W_s ds}  \;. \label{sec:eq13}
\end{equation}
For every fixed $s \in [0,T]$, $W_s$ and $W_T$, the latter is included in $X_0^T$ by its very definition, are Gaussian random variable jointly distributed. Therefore exploiting basic properties of Brownian motion we can recast them by means of a sum of independent random variable, namely
\begin{align*}
W_s &= Y \sim \mathcal{N}(0,s)\;, \\
W_T &= W_T-W_s +W_s = X+Y \;.
\end{align*}
In particular $X \sim \mathcal{N}(0,T-s)$ and it is independent with respect to $Y$. Thus \eqref{sec:eq13} reads

{\footnotesize
\begin{equation*}
\begin{split}
&\beta_6  \int_0^T  \int_{\R \times \R} \Indicator{\{x+y>-\sqrt{T}d_2\}}e^{x_0+\mu T}e^{\sigma_0(x+y)}y\frac{1}{\sqrt{2 \pi}}e^{\frac{x^2}{2(T-s)}} \frac{1}{\sqrt{2 \pi}}e^{-\frac{y^2}{2s}} ds = \\
 &=\beta_6   s_0 e^{rT}e^{-\frac{\sigma_0^2}{2}T} \int_0^T \int_{\R \times \R} \Indicator{\{x+y>-\sqrt{T}d_2\}}e^{\sigma_0(x+y)}y\phi(x;0;T-s)\phi(y,0,s)dxdyds\;,
\end{split}
\end{equation*}
}

and the claim follows.

\end{proof}

\subsubsection{Numerical results concerning the pricing formula in Prop. \ref{PRO:LinCorr}}

Let us consider the case of the B-S model corrected by a linear term given as in Prop. \ref{PRO:LinCorr} by $f(x) = \alpha_0 +\alpha_1 x$. We compute the first order correction of the price of an European call option with $\Phi(X^\epsilon_T) = (e^{X_T^\epsilon-K})_+$ as payoff function, according to Prop. \ref{PRO:LinCorr}.

Our aim is computing the expectation in \eqref{sec:eqExpect} in the present case. By the very definition of $X_0^T$ and $X_1^T$ and the form of $\Phi'$, it reads as
\begin{equation}
\begin{split}
&\Expectrn{\Indicator{D}e^{X_0^T}\beta_1 T}+\Expectrn{\Indicator{D}e^{X_0^T}\beta_2 T^2}+\Expectrn{\Indicator{D}e^{X_0^T}\beta_3W_T}\\
+& \Expectrn{\Indicator{D}e^{X_0^T}\beta_4W_T^2}+\Expectrn{\Indicator{D}e^{X_0^T}\beta_5 T W_T}-\Expectrn{\Indicator{D}e^{X_0^T}\beta_6 \int_0^T W_s ds} \label{sec:eqLin}
\end{split}
\end{equation}
Each random variable in the brackets is approximated by means of a { \it multi-element} PCE of degree $p=15$ and respectively by means of standard Monte Carlo methods, using $N=10000$ independent simulations of the random variable involved. 

The accuracy of {\it PCE} is represented by its absolute error, using as benchmark the analytical value coming from \eqref{EQN:PrezzoPoli}. Due to the {\it Law of Large Numbers}, the accuracy of MC-estimation of \eqref{EQN:ExpOptPrice} is provided by its standard error ($SE_{MC}$). Upon considering $N =10000$ realizations ($Y_j$) of the random variable $Y:=\Phi'(X_T^0)X_T^1$ inside the expectation in the r.h.s. of equation \eqref{EQN:ExpOptPrice}, let us compute
\begin{equation}
SE_{MC} = \frac{\hat{\sigma}}{\sqrt{N}}
\end{equation}
where $\hat{\sigma}^2 = \frac{1}{N-1} \sum_{j=1}^N \left(Y_j - \mu_{MC}\right)^2 $ and $\mu_{MC}= \frac{1}{N}\sum_{j=1}^N Y_j $.

The numerical values of the parameters involved are collected in Table \ref{tab:tab8}
\begin{table}[!h]
\centering
\begin{tabular}{c|cccccccc}
Parameters &$\alpha_0$ &$\alpha_1$ & r&$\sigma_1$ &K &T \\
\hline
Values  & 0.3 & 0.5 & 0.03 & 0.1 &  100 & 0.5\\
\end{tabular}
\caption{Numerical values of the parameters employed in further computations} \label{tab:tab8}
\end{table}  


The computations are made setting the parameters as in Table \ref{tab:tab8} and for a set of volatility values $\sigma_0 \in \{15 \%, 25 \%, 35 \%\}$ and for a set of increasing spot price $s_0 \in \{90,100,110\}$.

\begin{table}[!h]
\centering
\begin{tabular}{|c|c|ccc|ccc|}
\hline
 &&\multicolumn{3}{c}{$\epsilon = 0.1$}  & \multicolumn{3}{|c|}{$\epsilon = 0.01$} \\ 
\hline
& & Analytical & PCE & standard MC &Analytical & PCE & standard MC \\[1pt]
\hline 
\multirow{2}{*}{$\sigma_0 = 15 \%$} &
 Results & 1.45057 & 1.45049 &1.44774 &1.12927 & 1.12927 & 1.12870\\ 
 & Error &&7.9315e-05 & 8.2194e-03 && 7.9315e-06 &8.2548e-04 \\ 
\hline
\multirow{2}{*}{$\sigma_0 = 25 \%$} &
 Results & 3.82504 & 3.82488 &3.83225 &3.28856 & 3.28855 & 3.28849\\ 
 & Error &&1.5990e-04 & 1.1922e-02 && 1.5990e-05 &1.1596e-03 \\ 
\hline
\multirow{2}{*}{$\sigma_0 = 35 \%$} &
 Results & 6.44932 & 6.44905 &6.44766 &5.71657 & 5.71654 & 5.71482\\ 
 & Error &&2.7379e-04 & 1.5876e-02 && 2.7379e-05 &1.5413e-03 \\ 
\hline

\end{tabular}
\caption{Numerical values for PCE and MC estimation of equation \eqref{sec:eqExpect}, $s_0 = 90$, $\alpha_0 = 0.3$, $\alpha_1 = 0.5$, $\sigma_1 = 0.10 $, $r=0.03$, $K=100$ and $T=0.5$.} \label{tab:tab9}
\end{table}

\begin{table}[!h]
\centering
\begin{tabular}{|c|c|ccc|ccc|}
\hline
 &&\multicolumn{3}{c}{$\epsilon = 0.1$}  & \multicolumn{3}{|c|}{$\epsilon = 0.01$} \\ 
\hline
& & Analytical & PCE & standard MC &Analytical & PCE & standard MC \\[1pt]
\hline 
\multirow{2}{*}{$\sigma_0 = 15 \%$} &
 Results & 5.53650 & 5.53637 &5.53931 &5.03945 & 5.03944 & 5.04061\\ 
 & Error &&1.2763e-04 & 9.2641e-03 && 1.2763e-05 &9.3834e-04 \\ 
\hline
\multirow{2}{*}{$\sigma_0 = 25 \%$} &
 Results & 8.50577 & 8.50556 &8.52655 &7.83481 & 7.83479 & 7.83473\\ 
 & Error &&2.0666e-04 & 1.3297e-02 && 2.0666e-05 &1.3194e-03 \\ 
\hline
\multirow{2}{*}{$\sigma_0 = 35 \%$} &
 Results & 11.50225 & 11.50192 &11.49891 &10.63364 & 10.63361 & 10.63396\\ 
 & Error &&3.3318e-04 & 1.7426e-02 && 3.3318e-05 &1.7601e-03 \\

 \hline
\end{tabular}
\caption{Numerical values for PCE and MC estimation of equation \eqref{sec:eqExpect}, $s_0 = 100$, $\alpha_0 = 0.3$, $\alpha_1 = 0.5$, $\sigma_1 = 0.10 $, $r=0.03$, $K=100$ and $T=0.5$.} \label{tab:tab10}
\end{table}

\begin{table}[!h]
\centering
\begin{tabular}{|c|c|ccc|ccc|}
\hline
 &&\multicolumn{3}{c}{$\epsilon = 0.1$}  & \multicolumn{3}{|c|}{$\epsilon = 0.01$} \\ 
\hline
& & Analytical & PCE & standard MC &Analytical & PCE & standard MC \\[1pt]
\hline 

\multirow{2}{*}{$\sigma_0 = 15 \%$} &
 Results & 12.70933 & 12.70924 &12.72127 &12.37689 & 12.37688 & 12.37642\\ 
 & Error &&9.7072e-05 & 1.2457e-02 && 9.7072e-06 &1.2373e-03 \\ 
\hline
\multirow{2}{*}{$\sigma_0 = 25 \%$} &
 Results & 15.16320 & 15.16299 &15.16028 &14.53658 & 14.53656 & 14.53696\\ 
 & Error &&2.0462e-04 & 1.5486e-02 && 2.0462e-05 &1.5493e-03 \\ 
\hline
\multirow{2}{*}{$\sigma_0 = 35 \%$} &
 Results & 17.99767 & 17.99732 &18.01788 &17.10754 & 17.10750 & 17.10605\\ 
 & Error &&3.5427e-04 & 2.0558e-02 && 3.5427e-05 &1.9645e-03 \\

 \hline
\end{tabular}
\caption{Numerical values for PCE and MC estimation of equation \eqref{sec:eqExpect}, $s_0 = 110$, $\alpha_0 = 0.3$, $\alpha_1 = 0.5$, $\sigma_1 = 0.10 $, $r=0.03$, $K=100$ and $T=0.5$.} \label{tab:tab11}
\end{table}

\subsection{A correction given by a polynomial function and jumps}\label{SEC:PolJ}

In the present section we generalize the results obtained in the previous subsection \ref{SEC:PolCorr} adding a compensated Poisson random measure. In particular let us assume that the normal return of the asset price evolves according to eq. \eqref{EQN:NormRetuLevy} with a polynomial $f$. Then we have the following proposition.

\begin{proposition}\label{PRO:ApproxBSjump}
Let us consider the case of the B-S model with added compensated Poisson noise and corrected by a non-linear term given by \eqref{EQN:NormRetuLevy} with $f(x) = \sum_{i=0}^N \alpha_i x^i$, for some $\alpha_i \in \RR$, then the expansion coefficients for the solution $X^\epsilon_t$ of \eqref{EQN:NormRetuLevy} up to the second order are given by the system
{\footnotesize
\begin{equation}\label{EQN:32}
\begin{split}
X_t^0 &= x_0+ \mu t +\sigma_0  W_t, \quad \mbox{ with law }\quad \mathcal{N}\left (x_0+\mu t,\sigma_0^2 t\right );\\
X_t^1 & = \sum_{i=1}^N \tilde{K}_i (X^0_t)^{i+1}- \sum_{i=0}^N \int_0^t   K_i (X^0_s)^i ds + \sigma_1 \alpha_0 W_t -\lambda t \left (e^{\gamma + \frac{\delta^2}{2}}-1\right ) + \sum_{i=1}^{N_t} J_i\,; \\
X_t^2 &= \sum_{k=1}^{2N+1}C^1_k (X^0_t)^k - \sum_{k=1}^{2N+1}\int_0^t C^2_k (X^0_s)^k ds+ \sum_{i=1}^N \sum_{j=0}^N \int_0^t\int_0^s C_{i,j}^3 (X^0_s)^{i-1}(X_r^0)^j dr ds\\
&+ \sum_{i=1}^N \sum_{j=0}^N (X^0_t)^i\int_0^s C_{i,j}^4 (X_r^0)^j dr + \sum_{i=0}^{N-1} C^5_i \lambda \left (e^{\gamma + \frac{\delta^2}{2}}-1\right ) \int_0^t   s  \left (X_s^0\right )^{i} ds \\
&-\alpha_{i+1} \sigma_1t \left (X_t^0\right )^{i}\lambda \left (e^{\gamma + \frac{\delta^2}{2}}-1\right ) +\\
&+\int_0^t\sigma_1 \alpha_1 W_s ds\lambda \left (e^{\gamma + \frac{\delta^2}{2}}-1\right )-\sigma_1 t \alpha_1 W_t \lambda \left (e^{\gamma + \frac{\delta^2}{2}}-1\right )\\
&+ \sum_{i=2}^{N} \alpha_i \sigma_1 \lambda \left (e^{\gamma + \frac{\delta^2}{2}}-1\right ) \int_0^t \left (X_s^0\right )^{i} ds + \sigma_1 \alpha_1 W_t + \sum_{i=2}^N \sigma_1 \alpha_i \left (X_s^0\right )^{i}\sum_{i=1}^{N_t} J_i \\
&- \sum_{i=2}^N \sigma_1 \alpha_i \int_0^t\int_{\RR_0}\left (X_s^0\right )^{i}ds \sum_{i=1}^{N_t} J_i + \sum_{i=0}^{N-1} C^5_i \int_0^t \sum_{i=1}^{N_s} J_i\left (X_s^0\right )^{i}ds
\end{split}
\end{equation}}
where the constants are as in Prop. \ref{PRO:ApproxBS} and 
\[
C_i^5 =
\begin{cases}
\sigma_0^2 \sigma_1 \alpha_2 +2 \sigma_0 \sigma_1 \alpha_1 , \quad i=0\, ,\\
\sigma_1 \mu \alpha_{i+1}(i+1) + \frac{\sigma_0^2}{2}(i+2)(i+1) + 4 \sigma_0 \sigma_1 \alpha_{i+1}, \quad  i \not = 0, i \not= N\, ,\\
\alpha_N \sigma_1 N \mu + 2 \sigma_0 \sigma_1 N \alpha_{N+1}, \quad i \not= N\, ,\\
\end{cases}
\]
and
\[
\tilde{K}_i= \frac{\sigma_1}{\sigma_0}\frac{\alpha_i}{(i+1)}\, .
\]
\end{proposition}
\begin{proof}
The proof is analogous to the one in Prop. \ref{PRO:ApproxBS} taking into account the compensated Poisson random measure terms and applying It\^{o}'s lemma together with the stochastic Fubini theorem.
\end{proof}

\begin{proposition}\label{PRO:LinCorrJ}
Let us consider the particular case of $N=1$, i.e. a linear perturbation, namely $f(x) = \alpha_0 + \alpha_1 x$ in Prop. \ref{PRO:ApproxBSjump}. Then the terms up to the first order in equation \eqref{EQN:32} read
\begin{equation}\label{EQN:LinePrice}
\begin{split}
X^0_t &= x_0 + \mu t + \sigma_0 W_t\, ,\\
X^1_t &= \beta_1 t + \beta_2 t^2 + \beta_3 W_t + \beta_4 W_t^2 +\\
&+ \beta_5 t W_t - \int_0^t \beta_6 W_s ds-\lambda t \left (e^{\gamma + \frac{\delta^2}{2}}-1\right ) + \sum_{i=1}^{N_t} J_i\,,
\end{split}
\end{equation}

the constants being as in Prop. \ref{PRO:LinCorr}.

Also, the first order correction of the price of an European call option $\Phi$ with payoff given by $\Phi(X^\epsilon_T)= \left (e^{X^\epsilon_T}- K\right )_+$ (in the sense of remark \ref{REM:Smooth}) is explicitly given by
\begin{equation}\label{EQN:PrezzoPoli2}
\begin{split}
Pr^1(0;T) =Pr^1+ \epsilon Ts_0 N\left (d(1)\right ) \left (e^{\gamma + \frac{\delta^2}{2}}-1\right ) + \epsilon Ts_0 N\left (d(1)\right ) \delta \lambda\, ,
\end{split}
\end{equation}
where $Pr^1$ is the corrected fair price up to the first order as given in eq. \eqref{EQN:PrezzoPoli} and the notations are as above.
\end{proposition}
\begin{proof}
The proof is similar to the one in Prop. \ref{PRO:LinCorr}.
\end{proof}

\subsubsection{Numerical results concerning the pricing formula in Prop. \ref{PRO:ApproxBSjump}}

The $J_i$ are assumed to be independent and  normally distributed random variables
\begin{equation*}
J_i \sim \mathcal{N}(\gamma,\delta^2)\, , \, \mbox{ for all } \, i \in \{1,2,\dots,N_T\}\,, \qquad \gamma = 0.05, \qquad \delta = 0.02 \;,
\end{equation*}
and $\lambda = 2 $. In particular we are aiming at computing the expectation in \eqref{sec:eqExpect} for the model described in Prop. \ref{PRO:ApproxBSjump}. In the present case we have that this expectation in equal to

\begin{equation}
\begin{split}
&\Expectrn{\Indicator{D}e^{X_0^T}\beta_1 T}+\Expectrn{\Indicator{D}e^{X_0^T}\beta_2 T^2}+\Expectrn{\Indicator{D}e^{X_0^T}\beta_3W_T}\\
+& \Expectrn{\Indicator{D}e^{X_0^T}\beta_4W_T^2}+\Expectrn{\Indicator{D}e^{X_0^T}\beta_5 T W_T}-\Expectrn{\Indicator{D}e^{X_0^T}\beta_6 \int_0^T W_s ds} \\ +&K_2\Expectrn{\Indicator{D}e^{X_T^0}\lambda T \left( e^{\gamma+\frac{\delta^2}{2}}-1 \right)}+K_2\Expectrn{\Indicator{D}e^{X_T^0}\sum_{i=1}^{N_T}J_i}\;.\label{sec:eqLin3.1}
\end{split}
\end{equation}
By means of the independence of the jumps and $\mathbb{E} \left [\sum_{i=1}^{N_T}J_i\right ]= \delta \lambda T$, we can rewrite \eqref{sec:eqLin3.1} as
\begin{equation}
\begin{split}
&\Expectrn{\Indicator{D}e^{X_0^T}\beta_1 T}+\Expectrn{\Indicator{D}e^{X_0^T}\beta_2 T^2}+\Expectrn{\Indicator{D}e^{X_0^T}\beta_3W_T}\\
+& \Expectrn{\Indicator{D}e^{X_0^T}\beta_4W_T^2}+\Expectrn{\Indicator{D}e^{X_0^T}\beta_5 T W_T}-\Expectrn{\Indicator{D}e^{X_0^T}\beta_6 \int_0^T W_s ds}\\ +&K_2\Expectrn{\Indicator{D}e^{X_T^0}\lambda T \left( e^{\gamma+\frac{\delta^2}{2}}-1 \right)}+K_2 \lambda T \delta \Expectrn{\Indicator{D}e^{X_T^0}}\;.\label{sec:eqExp3}
\end{split}
\end{equation}
We shall then compute {\it multi-element} PCE-approximations for this expression.

The parameters are taken from Table \ref{tab:tab8} and the three spot prices, resp, volatilities, considered  are $s_0 \in \{90,100,110\}$, resp. $\sigma_0 \in \{15 \%, 25 \%, 35 \%\}$.

\begin{table}[!h]
\centering
\begin{tabular}{|c|c|ccc|ccc|}
\hline
 &&\multicolumn{3}{c}{$\epsilon = 0.1$}  & \multicolumn{3}{|c|}{$\epsilon = 0.01$} \\ 
\hline
& & Analytical & PCE & standard MC &Analytical & PCE & standard MC \\[1pt]
\hline 
\multirow{2}{*}{$\sigma_0 = 15 \%$} &
 Results & 1.45057 & 1.45049 &1.44774 &1.12927 & 1.12927 & 1.12870\\ 
 & Error &&7.9315e-05 & 8.2194e-03 && 7.9315e-06 &8.2548e-04 \\ 
\hline
\multirow{2}{*}{$\sigma_0 = 25 \%$} &
 Results & 3.82504 & 3.82488 &3.83225 &3.28856 & 3.28855 & 3.28849\\ 
 & Error &&1.5990e-04 & 1.1922e-02 && 1.5990e-05 &1.1596e-03 \\ 
\hline
\multirow{2}{*}{$\sigma_0 = 35 \%$} &
 Results & 6.44932 & 6.44905 &6.44766 &5.71657 & 5.71654 & 5.71482\\ 
 & Error &&2.7379e-04 & 1.5876e-02 && 2.7379e-05 &1.5413e-03 \\ 

  \hline
\end{tabular}
\caption{Numerical values for PCE and MC estimation of equation \eqref{sec:eqExpect}, $s_0 = 90$, $\alpha_0 = 0.3$, $\alpha_1 = 0.5$, $\sigma_1 = 0.10 $, $r=0.03$, $K=100$,$\lambda = 2$, $\gamma = 0.05$, $\delta =0.02$ and $T=0.5$.} \label{tab:tab13}
\end{table}

\begin{table}[!h]
\centering
\begin{tabular}{|c|c|ccc|ccc|}
\hline
 &&\multicolumn{3}{c}{$\epsilon = 0.1$}  & \multicolumn{3}{|c|}{$\epsilon = 0.01$} \\ 
\hline
& & Analytical & PCE & standard MC &Analytical & PCE & standard MC \\[1pt]
\hline 
\multirow{2}{*}{$\sigma_0 = 15 \%$} &
 Results & 5.53650 & 5.53637 &5.53931 &5.03945 & 5.03944 & 5.04061\\ 
 & Error &&1.2763e-04 & 9.2641e-03 && 1.2763e-05 &9.3834e-04 \\ 
\hline
\multirow{2}{*}{$\sigma_0 = 25 \%$} &
 Results & 8.50577 & 8.50556 &8.52655 &7.83481 & 7.83479 & 7.83473\\ 
 & Error &&2.0666e-04 & 1.3297e-02 && 2.0666e-05 &1.3194e-03 \\ 
\hline
\multirow{2}{*}{$\sigma_0 = 35 \%$} &
 Results & 11.50225 & 11.50192 &11.49891 &10.63364 & 10.63361 & 10.63396\\ 
 & Error &&3.3318e-04 & 1.7426e-02 && 3.3318e-05 &1.7601e-03 \\

 \hline
\end{tabular}
\caption{Numerical values for PCE and MC estimation of equation \eqref{sec:eqExpect}, $s_0 = 100$, $\alpha_0 = 0.3$, $\alpha_1 = 0.5$, $\sigma_1 = 0.10 $, $r=0.03$, $K=100$, $\lambda = 2$, $\gamma = 0.05$, $\delta =0.02$ and $T=0.5$.} \label{tab:tab14}
\end{table}

\begin{table}[!h]
\centering
\begin{tabular}{|c|c|ccc|ccc|}
\hline
 &&\multicolumn{3}{c}{$\epsilon = 0.1$}  & \multicolumn{3}{|c|}{$\epsilon = 0.01$} \\ 
\hline
& & Analytical & PCE & standard MC &Analytical & PCE & standard MC \\[1pt]
\hline 
\multirow{2}{*}{$\sigma_0 = 15 \%$} &
 Results & 12.70933 & 12.70924 &12.72127 &12.37689 & 12.37688 & 12.37642\\ 
 & Error &&9.7072e-05 & 1.2457e-02 && 9.7072e-06 &1.2373e-03 \\ 
\hline
\multirow{2}{*}{$\sigma_0 = 25 \%$} &
 Results & 15.16320 & 15.16299 &15.16028 &14.53658 & 14.53656 & 14.53696\\ 
 & Error &&2.0462e-04 & 1.5486e-02 && 2.0462e-05 &1.5493e-03 \\ 
\hline
\multirow{2}{*}{$\sigma_0 = 35 \%$} &
 Results & 17.99767 & 17.99732 &18.01788 &17.10754 & 17.10750 & 17.10605\\ 
 & Error &&3.5427e-04 & 2.0558e-02 && 3.5427e-05 &1.9645e-03 \\ 
 
  \hline
\end{tabular}
\caption{Numerical values for PCE and MC estimation of equation \eqref{sec:eqExpect}, $s_0 = 110$, $\alpha_0 = 0.3$, $\alpha_1 = 0.5$, $\sigma_1 = 0.10 $, $r=0.03$, $K=100$, $\lambda = 2$, $\gamma = 0.05$, $\delta =0.02$ and $T=0.5$.} \label{tab:tab15}
\end{table}

\section{Conclusions}
In this work we have focused our attention on the analysis of the small noise asymptotic expansions for particular  classes of local volatility models arising in finance. We have given  explicit expressions for the associated coefficients, along with  accurate estimates on the remainders. Furthermore we have provided a detailed numerical analysis, with accuracy comparisons, of the obtained results exploiting the standard Monte Carlo technique as well as the  so called Polynomial Chaos Expansion approach.
We would like to underline that our approach allows to consider, other than the well know Gaussian noise component, a realistic stochastic perturbation of jump type.

In a future work we plan to use the latter extension, along with the described asymptotic expansion techniques, to study particular types of implied volatilities models and further related functionals, as suggested by one of the anonymous reviewer. Such developments will be also the basis for an extensive calibration work on real financial data.

\section*{Acknowledgement}
The first author gratefully acknowledges the hospitality and the support received from the University of Verona in the framework of the Cooperint internationalization program, by the CIRM ({\it Centro Internazionale per la Ricerca Matematica}) funded by the FBK ({\it Fondazione Bruno Kessler}) and by the  University of Trento.
The second and the third author would like to thank the {\it Gruppo Nazionale per l'Analisi Matematica, la Probabilità e le loro Applicazioni} (GNAMPA) for the financial support that has  funded the present research within the project called {\it Set-valued and optimal transportation theory methods to model financial markets with transaction costs both in deterministic and  stochastic frameworks}.

\cleardoublepage


\begin{thebibliography}{99}
\bibitem[Albeverio {\it et al.}(2011)]{Alb3} S. Albeverio, L. Di Persio \& E. Mastrogiacomo (2011) Small noise asymptotic expansion for stochastic PDE's, the case of a dissipative polynomially bounded non linearity I, {\it Toh\^{o}ku Mathematical Journal}, {\bf 63}, 877--898.

\bibitem[Albeverio {\it et al.}(2016a)]{Alb2a} S. Albeverio, L. Di Persio, E. Mastrogiacomo \& B. Smii (2016a) A Class of L\'{e}vy Driven SDEs and their Explicit Invariant Measures, to be published in  {\it Potential Analysis}.

\bibitem[Albeverio {\it et al.}(2016b)]{Alb2b} S. Albeverio, L. Di Persio, E. Mastrogiacomo \& B. Smii (2016b) Invariant measures for SDEs driven by L\'{e}vy noise. A case study for dissipative nonlinear drift in infinite dimension, submitted.

\bibitem[Albeverio {\it et al.}(2012)]{AlbH} S. Albeverio, A. Hilbert \& V Kolokoltsov (2012) Uniform asymptotic bounds for the heat kernel and the trace of a stochastic geodesic flow, {\it Stochastics An International Journal of Probability and Stochastic Processes}, {\bf 84}, 315-333.

\bibitem[Albeverio {\it et al.}(2013)]{Alb} S. Albeverio \& B. Smii (2013) Asymptotic expansions for SDE's with small multiplicative noise, {\it Stochastic processes and their applications}, {\bf 125.3}, 1009-1031.

\bibitem[Albeverio {\it et al.}(2016c)]{Alb4} S. Albeverio \& V. Steblovskaya (2016c) Asymptotics of Gaussian integrals in infinite dimensions, paper in preparation.

\bibitem[Albeverio {\it et al.}(2006)]{Alb5} S. Albeverio, M. Schmitz, V. Steblovskaya \& K. Wallbaum (2006) A model with interacting assets driven by Poisson processes,  {\it Stochastic analysis and applications} {\bf 24.1} 241-261.

\bibitem[Andersen {\it et al.}(2013)]{Lip} L. Andersen \& A. Lipton (2012) Asymptotics for Exponential L\'{e}vy Processes and their Volatility Smile: Survey and New Results, {\it Int. J. Theor. Appl. Finance}, 16 135 0001.

\bibitem[Applebaum(2009)]{App} D. Applebaum (2009) {\it L\'{e}vy processes and stochastic calculus}, Cambridge Studies in Advanced Mathematics Vol. 116, Cambridge University Press, Cambridge.

\bibitem[Arnold(1974)]{Arn} L. Arnold (1974) {\it Stochastic differential equations: theory and applications},J. Wiley $\&$ Sons.

\bibitem[Bayer {\it et al.}(2014)]{Bay} C. Bayer \& P. Laurence (2014) Asymptotics beats Monte Carlo: The case of correlated local vol baskets {\it Communications on Pure and Applied Mathematics}, 67(10), 1618-1657.

\bibitem[Benarous {\it et al.}(2013)]{Ben} A. Benarous \& P. Laurence (2013) {\it Second Order Expansion for Implied Volatility in Two Factor Local Stochastic Volatility Models and Applications to the Dynamic} $\lambda-$ {\it Sabr Model}, Large Deviations and Asymptotic Methods in Finance, Springer International Publishing, 89-136.

\bibitem[Benhamou {\it et al.}(2009)]{BGM} E. Benhamou, E. Gobet, \& M. Miri (2009) Smart expansion and fast calibration for jump diffusions, {\it  Finance and Stochastics}, {\bf 13}, 563-589.

\bibitem[Black {\it et al.}(1973)]{BS} F. Black \& M. Scholes (1973) The Pricing of Options and Corporate Liabilities, {\it Journal of Political Economy}, {\bf 81},  637-654.

\bibitem[Bonollo {\it et al.}(2015a)]{BDPG} M. Bonollo, L. Di Persio \& G. Pellegrini (2015a) Polynomial Chaos Expansion approach to interest rate models, {\it Journal of Probability and Statistics}, {\bf 2015}.

\bibitem[Bonollo {\it et al.}(2015b)]{BDPG2} M. Bonollo, L. Di Persio \& G. Pellegrini (2015b) A computational spectral approach to interest rate models, arXiv preprint arXiv:1508.06236.

\bibitem[Breitung(1994)]{Bre} K. Breitung (1994) {\it Asymptotic approximations for probability integrals}, Springer.

\bibitem[Brigo {\it et al.}(2006)]{Bri} D. Brigo \& F. Mercurio (2006) {\it Interest rate models: theory and practice}, Springer Finance,Springer-Verlag, Berlin.

\bibitem[Carr {\it et al.}(2013)]{Car} P. Carr, T. Fisher \& J. Ruf (2013) Why are quadratic normal volatility models analytically tractable?, {\it SIAM Journal on Financial Mathematics}, {\bf 4.1} 185--202.

\bibitem[Cordoni {\it et al.}(2015)]{CDP} F. Cordoni \& L. Di Persio (2015) Small noise expansion for the L\'{e}vy perturbed Vasicek model, {\it International Journal of Pure and Applied Mathematics}, {\bf 98.2}.

\bibitem[Cox {\it et al.}(1985)]{CIR} J.C. Cox, J.E. Ingersoll \& S.A. Ross (1985) A theory of the term structure of interest rates, {\it Econometrica}, {\bf 53} 385-407.

\bibitem[Crepey (2004)]{Cre} S. Crepey (2004) Delta hedging vega risk, {\it Quant. finance}, {\bf 4} 559-579.

\bibitem[Crestaux {\it et al.}(2009)]{Crestaux} T. Crestaux, O.P. Le Ma\^{i}tre \& J.M. Martinez (2009) Polynomial chaos expansion for sensitivity analysis, {\it Reliability Engineering and System Safety}, {\bf 94.7} 1161-1172.

\bibitem[Ernst {\it et al.}(2012)]{Ern} O.G. Ernst, A. Muglera, H.J. Starkloffa \& E. Ullmann (2012) On the convergence of generalized polynomial chaos expansions, {\it ESAIM: Mathematical Modelling and Numerical Analysis}, {\bf 46.2} 317-339.

\bibitem[Filipovic (2009)]{Fil} D. Filipovic (2009) {\it Term-structure models}, Springer Finance,Springer-Verlag, Berlin.

\bibitem[Fouque {\it et al.} (2009)]{Fou} J.P. Fouque, G. Papanicolau \& R. Sircar (2000) {\it Derivatives in financial markets with stochastic volatility}, Cambridge University Press.

\bibitem[Friz {\it et al.}(2015)]{Gou} P.K. Friz, J. Gatheral, A. Guliashvili, A. Jacquier \& J. Teichman (2015) Large Deviations and Asymptotic Methods in Finance, {\it Springer Proceedings in Mathematics $\&$ Statistics}, {\bf 110}.

\bibitem[Fujii {\it et al.}(2012)]{Fuj} M.Fuji \& T. Akihiko (2012) Perturbative Expansion of FBSDE in an Incomplete Market with Stochastic Volatility, {\it The Quarterly Journal of Finance}, {\bf 2.03}.

\bibitem[Funahashi {\it et al.}(2016)]{Fun} H. Funahashi \& M. Kijima (2016). A chaos expansion approach for the pricing of contingent claims, {\it J. Computational Finance}, 13, 1-31.

\bibitem[Gardiner (2004)]{Gar} C.W. Gardiner (2004) {\it Handbook of stochastic methods for physics, chemistry and natural sciences}, Springer series in Synergetics, Springer-Verlag, Berlin.

\bibitem[Gatheral {\it et al.}(2012)]{Gat} J. Gatheral, E.P. Hsu, P. Laurence, C. Ouyang \& T.H. Wang (2012) Asymptotics of implied volatility in local volatility models, {\it Mathematical Finance}, {\bf 4} 591-620.

\bibitem[Giaquinta {\it et al.}(2000)]{Gia} M. Giaquinta \& G. Modica (2000) {\it An introduction to functions of several variables}, Birkh\"{a}user, Basel.

\bibitem[Gihman {\it et al.}(1972)]{Sko} I.I. Gihman \& A.V. Skorokhod (1972) {\it Stochastic differential equations}, Springer-Verlag, New York.

\bibitem[Gulisashvili (2012)]{Gu} A. Gulisashvili (2012) {\it Analytically Tractable Stochastic Stock Price Models}, Springer.

\bibitem[Imkeller {\it et al.}(2009)]{Imk} P. Imkeller, I. Pavlyukevich \& T. Wetzel (2009) First exit times for Lévy-driven diffusions with exponentially light jumps, {\it The Annals of Probability}, {\bf 37.2} 530-564.

\bibitem[Kallenberg (2006)]{Kal} O. Kallenberg (2006), Foundations of modern probability, Springer Science \& Business Media.

\bibitem[Kim {\it et al.}(1999)]{Kim} Y.J. Kim \& N. Kunitomo (1999) Pricing options under stochastic interest rates: a new approach, {\it Asia-Pacific Financial Markets}, {\bf 6.1} 49-70.

\bibitem[Kunitomo {\it et al.}(2003)]{Kun} N. Kunitomo \& A. Takahashi (2003) On validity of the asymptotic expansion approach in contingent claim analysis, {\it The Annals of Applied Probability}, {\bf 13.3} 914-952.

\bibitem[Kunitomo {\it et al.}(2001)]{Kun01} N. Kunitomo \& A. Takahashi (2001), The asymptotic expansion approach to the valuation of interest rate contingent claims, {\it Mathematical Finance} 11.1: 117-151.

\bibitem[Kusuoka {\it et al.}(2000)]{Kus} S. Kusuoka \& N. Yoshida (2000) Malliavin calculus, geometric mixing, and expansion of diffusion functionals, {\it Probability Theory and Related Fields}, {\bf 116.4} 457-484.

\bibitem[Le Maitre (2010)]{LeM} O.P. Le Maitre \& O.M. Knio (2010) {\it Introduction: Uncertainty Quantification and Propagation}, Springer Netherlands.

\bibitem[Lorig (2012)]{Lor} M. Lorig (2012) Local L\'{e}vy Models and their Volatility Smile, arXiv preprint  arXiv:1207.1630v1 .

\bibitem[L\"{u}tkebohmert (2004)]{Lu} E. L\"{u}tkebohmert (2004) An asymptotic expansion for a Black–Scholes type model, {\it Bulletin des sciences math\'{e}matiques}, {\bf 128.8} 661-685.

\bibitem[Mandelbrot {\it et al.}(2004)]{Mand} B. Mandelbrot \& R.L. Hudson O.P. (2004) {\it The Misbehavior of Markets: A fractal view of financial turbulence},  Basic books, New York.

\bibitem[Mandrekar {\it et al.}(2015)]{MaRu} V. Mandrekar \& B. R\"{u}diger (2015) {\it Stochastic Integration in Banach Spaces Theory and Applications}, Springer, Berlin.

\bibitem[Matsuoka {\it et al.}(2004)]{Mat} R. Matsuoka \& Y. Uchida (2004) A new computational scheme for computing Greeks by the asymptotic expansion approach, {\it Asia-Pacific Financial Markets}, {\bf 11.4} 393-430.

\bibitem[McKean (1969)]{McK} H.P. McKean (1969) {\it Stochastic integrals}, American Mathematical Soc.

\bibitem[Merton (1976)]{Mer} R. Merton (1976) Option pricing when underlying stock returns are discontinuous, {\it Journal of financial economics}, {\bf 3.1} 125-144.

\bibitem[Pagliarani {\it et al.}(2013)]{PPR} S. Pagliarani \& A. Pascucci \& C. Riga (2013), Adjoint expansions in local Lévy models, {\it SIAM Journal on Financial Mathematics} 4.1: 265-296.

\bibitem[Peccati {\it et al.}(2011)]{Peccati} G. Peccati \& M.S. Taqqu (2011) {\it Wiener Chaos: Moments, Cumulants and Diagrams: A Survey with Computer Implementation},  Springer Verlag.

\bibitem[Peszat {\it et al.}(2005)]{Pes} S. Peszat \& F. Russo (2005) Large-noise asymptotic for one-dimensional diffusions, {\it Bernoulli}, {\bf 11.2} 247-262.

\bibitem[Rudin (1986)]{Rud} W. Rudin (1986) {\it Real and complex analysis}, New York: McGraw-Hill Inc.

\bibitem[Shiraya {\it et al.}(2017)]{ST} K. Shiraya \& A. Takahashi (2017) An asymptotic expansion for local-stochastic volatility with jump models, {\it Stochastics} 89.1: 65-88.

\bibitem[Shreve (2004)]{Shr} S.E. Shreve (2004) {\it Stochastic calculus for finance II},  Springer Finance, Springer-Verlag, New York.

\bibitem[Takahashi (1999)]{Tak99} A. Takahashi (1999) An asymptotic expansion approach to pricing financial contingent claims, {\it Asia-Pacific Financial Markets} 6.2: 115-151.

\bibitem[Takahashi {\it et al.}(2014)]{Tak} A. Takahashi \& Y. Tsuruki (2014) A new improvement solution for approximation methods of probability density functions, {\it No. CIRJE-F-916. CIRJE, Faculty of Economics, University of Tokyo}.

\bibitem[Takahashi {\it et al.}(2012)]{Tak12} A. Takahashi \& T. Yamada (2012), An asymptotic expansion with push-down of Malliavin weights, {\it SIAM Journal on Financial Mathematics} 3.1: 95-136.

\bibitem[Uchida {\it et al.}(2004)]{UchYos} M. Uchida \& N. Yosida (2004) Asymptotic expansion for small diffusions applied to option pricing, {\it Stat. Inference Stoch. Process} {\bf 3} 189-223.

\bibitem[Wan {\it et al.}(2005)]{WanKarniadakis1} X. Wan \& G.E. Karniadakis (2005) An adaptive multi-element generalized polynomial chaos method for stochastic differential equations, {\it Journal of Computational Physics} {\bf 209.2} 617-642.

\bibitem[Wan {\it et al.}(2006)]{WanKarniadakis2} X. Wan \& G.E. Karniadakis (2006) Beyond Wiener–Askey expansions: handling arbitrary pdfs, {\it Journal of Scientific Computing} {\bf 27.1-3} 455-464.

\bibitem[Watanabe (1987)]{Wat87} S. Watanabe (1987), Analysis of Wiener functionals (Malliavin calculus) and its applications to heat kernels {\it The annals of Probability} : 1-39.

\bibitem[Wiener (1938)]{Wiener} N. Wiener (1938) The Homogeneous Chaos, {\it American Journal of Mathematics} {\bf 60} 897–936.

\bibitem[Yoshida (2003)]{Yos} N. Yoshida (2003) Conditional expansions and their applications, {\it Stochastic processes and their applications} {\bf 107.1} 53-81.

\bibitem[Yoshida (1992)]{Yos2} N. Yoshida (1992) Asymptotic expansion for statistics related to small diffusions {\it J. Japan Stat. Soc.} 22.2: 139-159.

\end{thebibliography}
\end{document}